\documentclass[9pt,a4paper]{article}

\usepackage{fullpage}
\usepackage[utf8]{inputenc}  
\usepackage[T1]{fontenc}

\usepackage{amsmath}%
\usepackage{amssymb}%
\usepackage{epsfig}%
\usepackage{amsthm}%
\usepackage{dsfont}%
\usepackage{natbib}
\usepackage{url}%
\usepackage{mathrsfs}%
\usepackage{braket}%
\usepackage{accents}
\usepackage{caption}
\usepackage{subcaption}
\usepackage{enumerate}
\usepackage{amsbsy}
\usepackage{mathtools}
\usepackage{breqn}

\mathtoolsset{showonlyrefs} 

\usepackage{graphicx}

\advance\voffset by 0pt
\advance\hoffset by 0pt
\renewenvironment{proof}[1][Proof]{\noindent\textbf{#1.} }{\ \rule{0.5em}{0.5em}}

\theoremstyle{plain}%
\newtheorem{theorem}{Theorem}
\newtheorem{lemma}{Lemma}

\newtheorem{proposition}{Proposition}
\newtheorem{remark}{Remark}
\theoremstyle{definition}%
\newtheorem{definition}{Definition}
\newtheorem{assumption}{Assumption}

\renewenvironment{proof}[1][Proof]{\textbf{#1.} }{\ \rule{0.5em}{0.5em}}

{\setcounter{section}{0}%
\numberwithin{equation}{section}%
}%
{}%

\newcommand{\sr}{\mathds{R}}%
\newcommand{\nr}{\mathds{N}}%
\newcommand{\Ep}{\textup{\textsf{E}}}%
\newcommand{\Pp}{\textup{\textsf{P}}}%
\newcommand{\Qp}{\textup{\textsf{Q}}}%
\newcommand{\Pip}{\Pi}%
\newcommand{\keyword}[1]{\noindent\emph{Keywords:} #1}%
\newcommand{\diff}[2]{\frac{\partial #1}{\partial #2}}%
\newcommand{\secdiff}[2]{\frac{\partial^2 #1}{\partial #2^2}}
\newcommand{\crossdiff}[3]{\frac{\partial^2 #1}{\partial #2 \partial #3}}
\newcommand{\pro}[1]{\left(#1_t\right)_{t\geq 0}}%

\newcommand{\expo}[1]{\exp\left( #1 \right)}

\newcommand{\rd}{\mathrm{d}}%

\newcommand{\cL}{{\mathscr L}}%
\newcommand{\cF}{{\mathscr F}}%
\newcommand{\cD}{{\mathscr D}}%
\newcommand{\cP}{{\mathscr P}}%
\newcommand{\cG}{{\mathscr G}}%

\newcommand{\cK}{{\mathcal K}}%
\newcommand{\CC}{{\mathcal C}}%
\newcommand{\LB}{{\underline{\mathcal R}}}%
\newcommand{\UB}{{\overline{\mathcal R}}}%

\newcommand{\bF}{{\mathbf F}}%
\newcommand{\bG}{{\mathbf G}}%

\newcommand{\CL}{{\mathcal L}}%
\newcommand{\CH}{{\mathcal H}}%
\newcommand{\vt}{\hat v}

\newcommand{\xL}{{\underline{x}}}%
\newcommand{\xU}{{\overline{x}}}%
\newcommand{\mL}{{\underline{m}}}%
\newcommand{\mU}{{\overline{m}}}%
\newcommand{\N}{\textsf{N}}%
\usepackage{todonotes}
\usepackage{soul}

\usepackage{hyperref}
\hypersetup{%
  colorlinks=true, 
 	linkcolor=blue!80!black, 
	citecolor=green!60!black, 
	filecolor=magenta,     
}

\numberwithin{equation}{section}

\title{Singular Control in Inventory Management with Smooth Ambiguity}
\author{Arnon Archankul\thanks{Department of Mathematics, University of York, United Kingdom. \ttfamily{arnon.archankul@york.ac.uk}} \hspace{0.5mm} and Jacco J.J. Thijssen\thanks{Department of Mathematics, University of York, United Kingdom. \ttfamily{jacco.thijssen@york.ac.uk}}}
\date{\today}

\begin{document}
\maketitle

\begin{abstract} 
    We consider singular control in inventory management under Knightian uncertainty, where decision makers have a smooth ambiguity preference over Gaussian-generated priors. We demonstrate that continuous-time smooth ambiguity is the infinitesimal limit of Kalman-Bucy filtering with recursive robust utility. Additionally, we prove that the cost function can be determined by solving forward-backward stochastic differential equations with quadratic growth. With a sufficient condition and utilising variational inequalities in a viscosity sense, we derive the value function and optimal control policy. By the change-of-coordinate technique, we transform the problem into two-dimensional singular control, offering insights into model learning and aligning with classical singular control free boundary problems. We numerically implement our theory using a Markov chain approximation, where inventory is modeled as cash management following an arithmetic Brownian motion. Our numerical results indicate that the continuation region can be divided into three key areas: (i) the target region; (ii) the region where it is optimal to learn and do nothing; and (iii) the region where control becomes predominant and learning should inactive. We demonstrate that ambiguity drives the decision maker to act earlier, leading to a smaller continuation region. This effect becomes more pronounced at the target region as the decision maker gains confidence from a longer learning period. However, these dynamics do not extend to the third region, where learning is excluded.
    
    \keyword{Inventory Model, Singular Control, Smooth Ambiguity, Knightian Uncertainty, Kalman-Bucy Filtering}
\end{abstract}

\section{Introduction}\label{sec:intro}

In this paper, we study a singular control problem in inventory management under smooth ambiguity. Our framework applies to settings in which a decision-maker (DM) must determine the optimal inventory level but cannot characterise inventory flow variability using a single, known probability measure. In other words, the DM is \emph{ambiguous}, about the true probabilistic structure governing future inventory demand.
We build on the seminal work of \citet{Karatzas1983-do} on singular control of discounted Brownian storage systems over finite horizons, extending it to a setting in which the DM views ambiguity as an unobservable process and learns about it through observations of the inventory dynamics. The DM's attitude to ambiguity is incorporated through a certainty equivalence of the standard expected utility. Our analysis focuses on the interplay between the inventory cost structure and the DM's attitude to ambiguity, under the assumption that the DM is ambiguity averse.

Managing inventory becomes a concern when a DM needs to (i) offload excess inventory, which incurs high holding costs, or (ii) restock inventory when levels are too low, potentially leading to penalties for delayed shipments. The former action can be executed, e.g., by offering the excess inventory at a promotional price, donating it, or eliminating it, each of which incurs a (proportional) cost. On the other hand, when the inventory replenishment is in demand the DM might need to issue partial refunds to customers, or request orders restock from another factory, both of which incur costs relative to their volume. These in fact create a trade-off which prompt the DM to seek the optimal policy, taking form of the upper and lower triggers (also known as reflecting barriers) of continual inventory demand, to minimise the overall operational costs.

There is a rich literature on inventory control. The first paper that addresses the problem belongs to \citet{Arrow1951-mm}, where inventory is modelled as simple discrete-time random variables with known reflecting barriers. This is also referred to in the literature as \emph{obstacle problems}. Later, \citet{Eppen1969-js} developed the model to handle unknown reflecting barriers, allowing the decision-maker to choose policies that minimise the cost function. This is known as \emph{singular control}. The continuous-time analogue of the obstacle problem was subsequently developed by \citet{Bather1966-wi,Vial1972-ip,Constantinides1976-rf}, where inventory demand follows Brownian motion. \citet{Michael_HARRISON1978-dq} was the first to address singular control of inventory in a continuous-time framework, which has since been extended to more generalised approaches by authors such as \citet{harrison1983instantaneous,Karatzas1983-do,bar1995explicit,bensoussan2005optimality,Dai2013-yz,Dai2013-xb}, among others. For a detailed overview of related research, see \citet{Har13}.

While prior literature has successfully applied inventory models in operations research, it often assumes that decision-makers operate under a single, known probability measure to guide inventory flow decisions. In practice, however, this assumption frequently breaks down, particularly when the demand distribution is unknown. This is evident in the newsvendor problem, a classic inventory model, where \citet{Ma2022-ew} show that ambiguity significantly increases decision error, as reflected by higher mean absolute percentage error from the normative benchmark (based on prior information) and lower expected profits. Ambiguity also leads to systematic under-ordering, especially for high-margin products, indicating a cautious bias. Crucially, the study reveals an interaction between risk and ambiguity: ambiguity has a stronger negative impact when stochastic variability (i.e., risk) is low, while high risk appears to dampen ambiguity's effect, highlighting the interplay between these two forms of uncertainty.

Further empirical support comes from \citet{Kocabiyikoglu2024-ng}, who examine how ambiguous demand and learning affect inventory decisions. Their study shows that managers often default to normative models, resulting in suboptimal outcomes, capturing at most 84\% of potential profit in high-margin settings, and only 51\% (or losses) in low-margin ones. They also find that order deviations grow under ambiguity, increasing backorder risks and penalty costs. Consistent with \citet{Ma2022-ew}, they observe that while ambiguity generally worsens performance, ``demand chasing'', updating decisions based on observed demand, can improve accuracy and profitability. Together, these findings highlight the critical impact of demand ambiguity and learning, emphasising the need for rigorous decision support models in inventory management under uncertainty.

In this paper, we develop a framework for singular control of inventory under learnable multiple priors, where ambiguity attitudes are captured through a second-order expected utility represented by a certainty equivalence operator, that is, a subjective weighted average over a set of priors. This approach follows the smooth ambiguity model \emph{\`{a} la} \citet{KlMaMu05,Klibanoff2009-zj}. While our motivation stems from inventory management, the results extend naturally to more general singular control settings.

The concept of ambiguity was introduced by \citet{knight1921risk} as a crucial aspect of uncertainty measurement. Knight differentiates between uncertainties with known probability distributions, which he terms as \emph{risk}, and those with unknown distributions, referred to as \emph{ambiguity}\footnote{In some literature, ambiguity is also referred to as \emph{Knightian uncertainty} in recognition of Knight's contributions.}. The DM behaviour under ambiguity was first studied by Ellsberg (1961) in what is now known as Ellsberg’s experiment. This experiment involves making bets on two urns, each containing 100 red or blue balls. In the first urn, the proportion of red to blue balls is known, whereas in the second urn, this proportion is unknown. The results indicated that most people prefer to bet on the first urn, demonstrating an aversion to ambiguity. In other words, people favour a risky bet over an ambiguous one.

Several attempts have been made to incorporate ambiguity into decision support theory. Notably, \citet{GiSch89} introduced the concept of maxmin utility and \citet{KlMaMu05} developed the smooth ambiguity model. Both models capture ambiguity preferences within the framework of subjective expected utility, albeit from different perspectives. The maxmin utility model shows that a rational DM can display ambiguity aversion by considering the worst-case prior out of a set of priors. Meanwhile, smooth ambiguity allows posterior learning under certainty equivalence to display ambiguity aversion.
In terms of Ellsberg's urn experiment, a DM using maxmin utility assumes the urn's composition changes with each draw. Conversely, a DM with smooth ambiguity preference believes the urn's composition remains fixed and can be learned through a sequence of draws.

In the maxmin approach, ambiguity preferences are modelled through the ``size'' of the set of priors, offering significant mathematical tractability. This feature makes the framework particularly attractive for a wide range of applications in finance, economics, and operations research, especially in continuous-time settings, as developed by \citet{ChEp02}. See, for instance, \citet{NiOz07, Trojanowska2010-zn, Thijssen2011-ag, Cheng2013-to, Hellmann2018-xs} for examples of such applications.

The first attempt to study (cash) inventory management under ambiguity is made by \citet{archankul_thijssen_ferrari_hellmann_2025}, who apply the maxmin utility framework to a singular control problem of firm cash reserves. They show that more ambiguity-averse DMs, represented by a larger set of priors, tend to exert control earlier, leading to shorter cash hoarding durations. This aligns with the empirical findings discussed earlier.

However, the maxmin approach cannot capture the managerial learning behaviours. A related contribution by \citet{Federico2023-pm} employs sequential detection to update the unknown trend, interpreted as a proxy for ambiguity, in a two-sided singular control problem under arithmetic Brownian motion. However, their model does not take into account the DM's attitude toward ambiguity, indicating the need for a framework that incorporates subjective utility alongside learning.

These limitations motivate the adoption of an alternative framework that accommodates learning under ambiguity. The smooth ambiguity model serves as a natural and flexible solution, as it allows for belief updating while accounting for ambiguity aversion. Unlike the maxmin approach, where the worst-case prior is always selected from the extreme boundary of the set, known as the \emph{upper-rim generator} (cf. \cite{ChEp02}), the smooth ambiguity model enables a weighted average across priors, offering a richer and more flexible representation of subjective beliefs.

The concept of smooth ambiguity was initially introduced by \citet{KlMaMu05} within a static framework. To adapt it for a dynamic setting, the same authors, in their subsequent work \citet{Klibanoff2009-zj}, employ a dynamic programming principle, by means of a recursive utility appearing in \citet{Epstein1989-ow}, to evaluate the cost function in each discrete time period, with the likelihood updated one-step-ahead through the Bayesian framework. However, complications arise when considering smooth ambiguity in continuous time, as highlighted by \citet{Skiadas2013-eb}. In continuous time, the role of smooth ambiguity diminishes over time. In other words, in the original model proposed by \citet{Klibanoff2009-zj}, short-term decision-making neglects ambiguity aversion.

Later, \citet{gindrat2011smooth} demonstrated that it is possible to retain smooth ambiguity in continuous time by introducing a time increment dependence into a certainty equivalence operator between each one-step-ahead decision. In the same year, \citet{Hansen2011-ix} established a connection between the model proposed by \citeauthor{gindrat2011smooth} and proposed a backward stochastic differential equation (BSDE) representation. They accomplished this by assuming that the probability density generator is an unobservable \emph{Gaussian-generated} process and employed the Kalman-Bucy filter (cf. \citet{liptser2013statisticsII}) to give the best estimate under the realised information. They then applied a logarithmic transformation (cf. \citet{fleming2006controlled}) to convert \citeauthor{Klibanoff2009-zj}'s utility function into the BSDE.

In our contribution, we build upon the idea presented by \citet{Hansen2011-ix} and provide a rigorous argument by establishing the existence and uniqueness of such a BSDE. This BSDE is shown to be of a particular type, with a driver exhibiting quadratic growth in the diffusion term. By contemplating the theory associated with this type of BSDE, as discussed in \citet{Kobylanski2000-zw,Zhang2017-jp}, we demonstrate the possibility of integrating it with singular control under certain regularity conditions. Subsequently, we establish a connection between the BSDE and the Hamilton–Jacobi–Bellman (HJB) equation, albeit in a weak sense\footnote{Here, when we refer to the weak sense, we are specifically alluding to the viscosity solution of the HJB equation, as described in \citet{crandall1983viscosity}, since we only have prior knowledge of the continuity of the BSDE, but not its smoothness.}. We employ the classical derivation method of \citet{harrison1983instantaneous} to suggest an optimal control policy within the framework of singular control under smooth ambiguity.

It is important to note that using the Kalman–Bucy filter results in two-dimensional diffusions: one associated with the inventory-controlled process and the other with the posterior estimate. In terms of applicability, this is problematic, since only the inventory flow can be physically observed, not the posterior fluctuations. To address this, we employ the change-of-coordinates technique introduced by Johnson and Peskir (2017)\footnote{We refer to \citet{De_Angelis2020-sr, Federico2023-pm, Basei2024-pq} for the application of coordinate transformation in finance and economics.}. This method reduces the diffusion dimensions to one, solely with respect to the inventory-controlled process, while the other transforms into a deterministic process with bounded variations. We demonstrate that these processes are coupled in a mean-reversion-like manner, providing further insight into the filtering problem of the Gaussian-generated hidden variable, particularly under smooth ambiguity adjustment.

The structure of our paper is as follows. In Section \ref{sec:model}, we formulate the inventory model, transitioning from singular control to the Kalman–Bucy filter. Section \ref{sec:vf} introduces a value function that incorporates smooth ambiguity. Section \ref{sec:derivation} presents a heuristic derivation of the optimal control policy by solving a Hamilton–Jacobi–Bellman (HJB) equation, and Section \ref{sec:viscosity} establishes the viscosity property of this solution. Section \ref{sec:coor tran} discusses the change-of-coordinates technique. In Section \ref{sec:comstat}, we use Markov chain approximation to numerically solve the HJB equation and perform comparative statics, before concluding the paper in Section \ref{sec:conclusion}.

\section{Model formulation}\label{sec:model}
    Fix $T<\infty$. Let $ (\Omega, \cF, \bF=\pro{\cF}, \Pp)$ be a  complete filtered probability space, with $\cF\supseteq\cF_{T}$. The associated conditional expectation operator is denoted by $\Ep^{\Pp}_{\cG_t}[\cdot]$.   
    Suppose that $\alpha:\sr\to\sr$ and $\sigma:\sr\to\sr^+$ are $\cF_t-$measurable functions such that
\begin{align}
    |\alpha(x)| +|\sigma(x)|&\leq C_0(1+|x|)\;\; \text{for all }x\in \sr \label{eq: suff con a}\\
    |\alpha(x)-\alpha(y)| + |\sigma(x) -\sigma(y)| &\leq C_0|x-y|\;\;\;\;\text{for all } x,y\in \sr\label{eq: suff con s}
\end{align}
for some $C_0\in \sr$. We assume that the \emph{inventory process} follows a time-homogeneous diffusion, $X\triangleq\pro{X}$, which is the unique strong solution to the stochastic differential equation (SDE),
\begin{equation}\label{eq: X}
  \rd X_t = \alpha(X_t) \rd t+\sigma(X_t)\rd B_t,\quad X_0 = x,\quad  \Pp-\text{a.s.},
\end{equation}
 where $B=\pro{B}$ is a standard $\Pp$-Brownian motion, adapted to $\bF$.

A \emph{control policy} is a pair of bounded variation processes $A\triangleq (A^-,A^+)$, where $A^-$ and $A^+$ are $\bF-$adapted, bounded, non-decreasing, and non-negative. These processes are associated with increases and decreases, respectively, of $X$ at times at which control is exerted. A control policy $A$ is said to be \emph{feasible} if for all $x\in \sr$, there exists a unique $X^{A}$ strongly solves 
\begin{align}\label{eq: XA}
    \rd X^{A}_t = \alpha(X^A_t) \rd t+\sigma(X^A_t)\rd B_t +\rd A \quad X_0 = x,\quad  \Pp-\text{a.s.},
\end{align}
and if there exist $\overline C>0$ and $\underline C <0$ such that
\begin{align}
  \Pp\left(\sup_{t\geq 0}X^A_t<\overline C, \inf_{t\geq 0}X^A_t>\underline C\right) &= 1
\end{align}
The set of feasible control policies is denoted by $\cD$. We call $X^A$ a \emph{controlled inventory process}. Note that there is $\epsilon>0$ such that
\begin{align}\label{cond:never intervene}
    \Pp\left(\sup_{t\geq 0}X^0_t<-\epsilon, \inf_{t\geq 0}X^0_t>\epsilon\right)>0,
\end{align}
where $X^0$ is an uncontrolled process, i.e., $A=0$. This ensures that to \textit{never intervene} is a feasible control policies, i.e., $0\in \cD.$

Ambiguity arises when a decision maker (DM) is unable to establish a single probability model. This inability might stem from a lack of information about the historical inventory flow evolution, or even if such information exists, the DM cannot agree on a single probability measure. Therefore, the DM contemplates a \emph{set of prior} denoted as $\cP$, which comprises a collection of probability measures that are equivalent to the \emph{reference prior} $\Pp$. The baseline prior represents the DM's best subjective estimate of the model. 


Applying the Radon-Nikodym theorem, we establish that, for a given measure $\Qp$ belonging to the set $\cP$, there exists an $\cF_T$-measurable function $\frac{\rd \Qp}{\rd \Pp}\Big|\cF_T$ known as the Radon-Nikodym derivative
In this paper, we assume that $\cP$ is generated by a set of density generators denoted by $\Theta$ via the Girsanov theorem (cf. \cite{karatzas1991brownian}). That is to say, there is an $\bF$-adapted process $\theta = \pro{\theta}\in\Theta$ such that
for each $\Qp\in\cP$ so that the Radon-Nikodym derivative $\frac{\rd \Qp}{\rd \Pp}\Big|\cF_T$ is given by the following $\Pp$-martingale process:
\begin{align}
\frac{\rd \Qp}{\rd \Pp}\Big|\cF_T = \expo{ -\int_0^T \theta_t \rd B_t - \frac{1}{2}\int_0^T\theta_t^2 \rd t }.
\end{align}
We refer to $\theta$ as the \emph{model}. Moreover, for all $\theta\in\Theta$, we assume that $\theta_t\in \cK\subseteq \sr$ a.s.\footnote{A spceial case of $\cK$ is the closed interval $[-\kappa,+\kappa]$, referred to as $\kappa$--ignorance, introduced by \citet{ChEp02}.} for any $t\geq 0$, and that $0\in \cK$ to ensure that $\Pp\in\cP$. Since $\frac{\rd \Qp}{\rd \Pp}\Big|\cF_T$ is a $\Pp$-martingale, from now on we denote for notional convenience by
\begin{align*}
\frac{\rd \Qp}{\rd \Pp}\Big|^{(t,s)}_{\cF_t}\triangleq\frac{\rd \Qp}{\rd \Pp}\Big|\cF_s\bigg/\frac{\rd \Qp}{\rd \Pp}\Big|\cF_t, \quad \textup{for any } s\in[t,T].
\end{align*}
Consequently, as shown in \cite{karatzas1991brownian}, the process $B^{\theta}=\pro{B^{\theta}}$ defined by
\begin{align}
B_t^{\theta} \triangleq B_t + \int_0^t\theta_s \rd s
\end{align}
is a $\Qp$-Brownian motion. Thus, under the measure $\Pp$, the controlled inventory process $X^{\theta,A}$ follows
\begin{align}\label{eq:inventory p theta}
\rd X^{\theta,A}_t &=\alpha(X^{\theta,A}_t)\rd t+\sigma(X^{\theta,A}_t)\rd B^{\theta}_t + \rd A\nonumber\\
&=\left(\alpha(X^{\theta,A}_t) -\sigma(X^{\theta,A}_t)\theta_t\right)\rd t+\sigma(X^{\theta,A}_t)\rd B_t + \rd A, \quad X^{\theta,A}_0 = x, 
\end{align}
admitted a unique strong solution. Moreover, the fact that $\Qp$ and $\Pp$ share the same null set ensures that $(L,U)\in\cD$ under $\Qp$. In other words, $(L,U)$ is also adapted in accordance with ambiguity. For brevity, we write $X^{0,L,U}=X^A $.

Given the inherent nature of ambiguity, it is reasonable to acknowledge that, due to insufficient information or the questionable reliability of the available information, the DM is assumed to treat the model $\theta\in\Theta$ as an unobservable process. Nevertheless, the model can be inferred from a stream of information $\bG$, generated by an observable process, which, in this context, is the controlled inventory process $X^{\theta,A}$. Here, $\bG=\pro{\cG}$ is such that $\cG_t\subseteq\cF_t$ for all $t\geq 0$. In other words, $\theta_t$ is $\cF_t$-measurable but not $\cG_t$-measurable. According to filtering theory (see \cite{liptser2013statisticsI}, Chapter 12), $\theta_t$ can be estimated by its conditional expectation with respect to $\bG$. Under the reference $\Pp$, we represent this estimator as a $\cG_t$-measurable process $$M_t^{\Pp}\triangleq\Ep^{\Pp}_{\cG_t}[\theta_t]_{t\geq0},$$ with the mean square error expressed as $$S_t\triangleq\Ep^{\Pp}_{\cG_t}[(\theta_t-M^{\Pp}_t)^2]_{t\geq0}.$$
As a result, the Radon-Nikodym derivative under a Girsanov kernel $M^{\Pp}$ is
\begin{align}
\frac{\rd \Qp}{\rd \Pp}\Big|\cG_T = \expo{ -\int_0^T M^{\Pp}_t \rd B_t - \frac{1}{2}\int_0^T(M^{\Pp}_t)^2 \rd t }.
\end{align}

Throughout this paper we consider the special case where $\theta$ is a \emph{constant}, i.e. $\rd \theta_t =0$, and has predetermined (believed) to be \emph{normal distributed} with mean $M^{\Pp}_0\in\sr$ and variance $S_0>0$, i.e. $\theta_0\sim\N(M^{\Pp}_0,S_0)$. This also means that $\theta_t\in\sr$ a.s. for any $t\geq0$, i.e. $\cK=\sr$. By these assumptions, it is thus implied by Kalman-Bucy filtering (cf. \citet{liptser2013statisticsII}) that there exists a strong unique solution to the process $R^{M^{\Pp},L,U} = (X^{M^{\Pp},A}_t,M^{\Pp}_t,S_t)_{t\geq 0}$ where
\begin{align}
    \rd X^{M^{\Pp},A}_t &= \phantom{-}(\alpha(X^{M^{\Pp},A}_t) -\sigma(X^{M^{\Pp},A}_t)M^{\Pp}_t )\rd t+\sigma(X^{M^{\Pp},A}_t)\rd \overline B_t + \rd A_t\\
    \rd M^{\Pp}_t &= \phantom{-}S_t\rd \overline B_t \\
    \rd S_t &= -(S_t)^2 \rd t\label{eq:S}
\end{align}
and  $\overline B\triangleq\pro{\overline B}$ is so-called an \emph{innovation process} solving
\begin{equation}\label{eq: innovation process}
    \overline B_t =\frac{1}{\sigma(X^{\theta,A}_t)}\int_0^t \left( X^{\theta,A}_s - (\alpha(X^{\theta,A}_s)-\sigma(X^{\theta,A}_s)M^{\Pp}_s) \right)\rd s
\end{equation}
which also is a $\Pp$--Brownian motion. We call $X^{M^{\Pp},A}, M^{\Pp}$ and $S$ a \textit{filtered controlled inventory process, belief process} and \textit{belief variance}, respectively. 

Note that the variance of the belief process is given by $S(t) = \frac{s}{1+st},t\geq 0$. This expression reveals that the variance of the belief increment decreases monotonically over time and tends to zero as $t\rightarrow\infty$. In other words, the longer the learning process continues, the more precise the belief becomes, implying that the hidden information is eventually revealed. Consequently, the terminal belief $M^{\Pp}_{\infty}$ becomes an \emph{absorbing} state of the belief process. 

Furthermore, the tail probability of the belief process can be computed as
\begin{align*}
    \Pp\left(|M_t-m|>h\mid M_0=m\right) = 2\left(1-\Phi\left(h\sqrt{\frac{1+st}{s^2t}}\right)\right)
\end{align*}
where $\Phi$ is the standard normal cumulative distribution function. This expression indicates that the likelihood of the belief deviating significantly from its reference point $M^{\Pp}_0=m$ decreases as $h$ increases.  Thus, it becomes increasingly unlikely for such tail events to be the absorbing states of the belief process.

Based on this observation, we assume that the DM has a predetermied \emph{confidence interval}, denoted by $[\mL, \mU]=[m-h,m+h]\subset \sr$, which is assumed to contain the absorbing state with high probability. Accordingly, when the belief process reaches the boundaries $\mL$ or $\mU$, the DM endogenously intervenes to confine the belief within the interval $[\mL, \mU]$. 

Formally, this intervention is modelled via a bounded variation process $D\triangleq D^-+D^+$, where $D^-$ and $D^+$ are $\bG-$adapted, bounded, non-decreasing, non-negative. They increase only when the belief process $M^{\Pp}_t$ hits the lower boundary $\mL$ or upper boundary 
$\mU$, respectively, for any $t\in[0,T]$. In other words, $D$ ensures that 
\begin{align}
  \Pp\left(\sup_{t\geq 0}M^{\Pp}_t\leq \mU, \;\inf_{t\geq 0}M^{\Pp}_t\geq \mL \right) &= 1.
\end{align}

 As a result, our setting becomes a mixed singular-obstacle control problem and so the inventory flow satisfies
\begin{align}
    \rd X^{M^{\Pp},A}_t &= \phantom{-}(\alpha(X^{M^{\Pp},A}_t) -\sigma(X^{M^{\Pp},A}_t)M^{\Pp}_t )\rd t+\sigma(X^{M^{\Pp},A}_t)\rd \overline B_t + \rd A_t,\\
    \rd M^{\Pp}_t &= \phantom{-}S_t\rd \overline B_t+ \rd D_t\\
    \rd S_t &= -(S_t)^2 \rd t.
\end{align}

\section{Optimisation problem under ambiguity}\label{sec:vf}
Now we introduce a discounted cumulative cost of inventory holding from time $t$ to $T$:
\begin{align}\label{def: discounted cumulative cost criteria}
    K(t,X^{M^{\Pp},A}_t)\triangleq\int_t^T e^{-\rho(s-t)}\left(f(X^{M^{\Pp},A}_s)\rd s +\ell \rd A^-_s + u \rd A^+_s \right)+ e^{-\rho(T-t)}\xi(X^{M^{\Pp},A}_T).
\end{align} 
Here $f:\sr\to\sr^+$, and $\xi:\sr\to\sr^+$ are Borel measurable mappings, which determine inventory holding cost and bounded terminal cost, respectively. Constants $\ell>0$ and $u>0$ are proportional cost when control on $X^{M^{\Pp},A}$ is exerted, and $\rho>0$ is a DM's discounted rate. For some $h>0$ such that $t+h\leq T$ for all $t\in[0,T]$, we define a linear operator $\Pi_{t+h}$ associated cumulative cost criterion \eqref{def: discounted cumulative cost criteria} at time $t$ by
\begin{align}\label{def: linear operator}
    \Pip_{t+h}K(t,X^{M^{\Pp},A}_t)=\int_t^{t+h} e^{-\rho(s-t)}\left(f(X^{M^{\Pp},A}_s)\rd s +\ell \rd A^-_s + u \rd A^+_s \right)+ K(X_{t+h}^{A}).
\end{align}
Therefore, for any feasible control policy $A\in\cD$, the standard expected cost of inventory holding of is 
\begin{align}\label{def:standard expected utility}
    \Ep^{\Pp}_{\cG_t}\left[ \Pip_{t+h}K(t,X^{M^{\Pp},A}_t)\right].
\end{align}



In this paper, we begin by assuming that the decision maker (DM) exhibits ambiguity-sensitive preferences, modelled within the smooth ambiguity framework introduced by \citet{KlMaMu05,Klibanoff2009-zj}. Specifically, the DM evaluates the cumulative cost criterion using a certainty equivalence operator, which distorts the standard expectation \eqref{def:standard expected utility} based on the information available at each point in time. This is captured by a \emph{convex function} $\psi_h$, through which the DM recursively evaluates the cost associated with any admissible policy $A \in \cD$ via the following recursive functional:
\begin{align}\label{def: distorted cost function}
    K^h(t,X^{M^{\Pp},A}_t) \triangleq\psi_h^{-1} \Ep^{\Pp}_{\cG_t}\Bigg[\psi_h\left(\Pip_{t+h}K(t,X^{M^{\Pp},A}_t)\right) \Bigg]
\end{align}
for $0\leq t<t+h\leq T<\infty$. The convexity of the function $\psi_h$ naturally distorts the standard expectation: when $\psi_h$ is decreasing, it biases the evaluation toward lower values; when increasing, it biases toward higher values. To facilitate a more tractable analysis, we adopt a time-increment-dependent exponential disutility specification for $\varphi_h$, as proposed by \citet{Hansen2011-ix} (see also, \citet{Hansen2018-yc,Hansen2022-ob} for a similar treatment). Specifically, we define 
\begin{align}\label{def: concave utility} 
\psi_h(x) \triangleq \exp\left( \frac{\gamma x}{h(1 - \rho h)} \right), \quad \textup{for } x \in \mathbb{R},  h > 0, 
\end{align} 
where the parameter $\gamma$ governs the DM's attitude toward ambiguity. For a fixed $h > 0$, the function $\varphi_h$ is decreasing when $\gamma < 0$, reflecting \textit{ambiguity aversion}; increasing when $\gamma > 0$, reflecting \textit{ambiguity seeking}; and approaches linearity as $\gamma \to 0$, corresponding to \textit{ambiguity neutrality}. The presence of $h$ is crucial to maintain the impact of ambiguity in a decision making within an interval $[t,t+h]$. Without it, ambiguity diminishes with time, which means that \eqref{def: distorted cost function} almost surely becomes a linear conditional expectation~\eqref{def:standard expected utility}, as demonstrated by \citet{Skiadas2013-eb}. To compactify the argument, we assume, for the rest of this paper, that the DM is \underline{\textit{ambiguity averse}}. The results for the ambiguity-seeking case follow analogously.

We show in the following section that in limit $h\downarrow0$ one can express the cost function \eqref{def: distorted cost function} by the robust control problem:
\begin{align}\label{def: distorted cost function 2}
\sup_{\Pp\in\cP}\Ep^{\Qp}_{\cG_t}\left[ K(t,X^{M^{\Pp},A}_t) + \Psi_{t,T}(\Pp\parallel\Qp)\right],
\end{align}
where $\Psi_{t,T}(\Pp\parallel\Qp)$ represents the \textit{statistical distance} associated with changing the probability measure from $\Pp$ to $\Qp$ at time $T$ given the information known at time $t$, known in the literature as \textit{relative entropy} or \textit{Kullback–Leibler divergence}. 

Hence, within the framework of a singular control with smooth ambiguity, the DM's objective is to determine an optimal policy $A$ that minimises the cost function \eqref{def: distorted cost function 2}. Therefore, the \textit{value function}, denoted as the process $V\triangleq\pro V$, is defined as:
\begin{align}\label{def: value function}
V_{t}\triangleq \inf_{A\in\cD}\sup_{\Qp\in\cP}\Ep^{\Qp}_{\cF_t}\left[ K(t,X_t^{M^{\Pp}_t,A}) + \Psi_{t,T}(\Pp\parallel\Qp)\right].
\end{align}
To analyse this value function, we begin by deriving an explicit form of the relative entropy. Then we demonstrate that solving the problem is equivalent to  solving a system of forward-backward stochastic differential equations (FBSDEs), given that $\theta_0$ has a Gaussian distribution. We then establish  sufficient conditions for the optimal conditions of the singular control. Finally, we establish a verification theorem to provide a unique viscosity solution of a Hamilton-Jacobi-Bellman (HJB) equation which gives a deterministic representation of the system of FBSDEs.

\begin{lemma}[Logarithmic Transformation]\label{lem: log-tran}
    The conditional expectation~\eqref{def: distorted cost function} satisfies
    \begin{align}\label{def: robust control}
        K^h(t,X_t^{M^{\Pp},A})\geq\sup_{\Qp\in\cP}\Ep^{\Qp}_{\cG_t}\Bigg[\Pip_{t+h}K(t,X^{M^{\Pp},A}_t) + \frac{h(1-\rho h)}{\gamma}\log\left(\frac{\rd \Pp}{\rd \Qp}\Big|^{(t,t+h)}_{\cG_t} \right)\Bigg| \cG_t\Bigg],
    \end{align}
    where $h>0$ is chosen to be small enough that $t+h<t+2h<\ldots\leq t+nh=T$, for some $n\in \nr$. Moreover, if we take $h\to 0$, then 
    \begin{align}\label{def: robust control 2}        
         J^{A}_t\triangleq \lim_{h\downarrow0}K^h(t,X_t^{M^{\Pp},A})&\geq\sup_{\Qp\in\cP}\Ep^{\Qp}_{\cG_t}\Bigg[K(t,X^{M^{\Pp},A}_t)+\int_t^{T} e^{-\rho(s-t)} \frac{1}{\gamma}\log\left(\frac{\rd \Pp}{\rd \Qp}\Big|^{(t,s)}_{\cG_t} \right)\rd s \Bigg].
    \end{align}
\end{lemma}
\noindent\begin{proof}[Proof of Lemma \ref{lem: log-tran}]
    Let $\Qp\in\cP$. We then have that
    \begin{align*}
         &\Ep^{\Pp}_{\cG_t}\Bigg[\expo{\frac{\gamma}{h(1-\rho h)}\Pip_{t+h}K(t,X^{M^{\Pp},A}_t)}\Bigg]\\
         &=\Ep^{\Qp}_{\cG_t}\left[\frac{\rd \Pp}{\rd \Qp}\Big|^{(t,t+h)}_{\cG_t}\expo{\frac{\gamma}{h(1-\rho h)}\Pip_{t+h}K(t,X^{M^{\Pp},A}_t)} \right] \\
        &=\Ep^{\Qp}_{\cG_t}\Bigg[\exp\Bigg(\frac{\gamma}{h(1-\rho h)}\Bigg(\Pip_{t+h}K(t,X^{M^{\Pp},A}_t) + \frac{h(1-\rho h)}{\gamma}\log\left(\frac{\rd \Pp}{\rd \Qp}\Big|^{(t,t+h)}_{\cG_t} \right)\Bigg)\Bigg)\Bigg]\\
        &\geq \exp\Bigg(\Ep^{\Qp}_{\cG_t}\Bigg[\frac{\gamma}{h(1-\rho h)}\Bigg(\Pip_{t+h}K(t,X^{M^{\Pp},A}_t)+ \frac{h(1-\rho h)}{\gamma}\log\left(\frac{\rd \Pp}{\rd \Qp}\Big|^{(t,t+h)}_{\cG_t} \right) \Bigg) \Bigg]\Bigg),
    \end{align*}
    from which we deduce that 
    \begin{align*}
        K^h(t,X_t^{M^{\Pp},A}) &\geq\sup_{\Qp\in\cP}\Ep^{\Qp}_{\cG_t}\Bigg[\Pip_{t+h}K(t,X^{M^{\Pp},A}_t)+ \frac{h(1-\rho h)}{\gamma}\log\left(\frac{\rd \Pp}{\rd \Qp}\Big|^{(t,t+h)}_{\cG_t} \right)  \Bigg],
    \end{align*}
    since $\Pp$ is arbitrary.
    By the mean value theorem (MVT), there exists $\delta>0$ such that
    \begin{align}\label{eq: linear KL}
         h(1-\rho h)\log\left(\frac{\rd \Pp}{\rd \Qp}\Big|^{(t,t+h)}_{\cG_t}\right)  &=\int_t^{t+h+\delta}(1-\rho(s-t))\log\left(\frac{\rd \Pp}{\rd \Qp}\Big|^{(t,s)}_{\cG_t}\right)  \rd s \nonumber\\
         &\geq \int_t^{t+h}(1-\rho(s-t))\log\left(\frac{\rd \Pp}{\rd \Qp}\Big|^{(t,s)}_{\cG_t}\right)  \rd s
    \end{align}
     By definition of the linear operator $\Pip_{t+h}$ and \eqref{eq: linear KL}, it follows that
     \begin{align}\label{eq: linear KL2}
             \sum_{k=1}^n\log\left(\frac{\rd \Pp}{\rd \Qp} (t + (k-1)h,t+ kh)\right) 
                &\geq \frac{1}{ h(1-\rho h)} \sum_{k=1}^n \int_{t + (k-1)h}^{t+kh}(1-\rho(s-t))\log\left(\frac{\rd \Pp}{\rd \Qp}\Big|^{(t,s)}_{\cG_t}\right)  \rd s\nonumber \\
             &=\frac{1}{ h(1-\rho h)} \int_t^T (1-\rho(s-t))\log\left(\frac{\rd \Pp}{\rd \Qp}\Big|^{(t,s)}_{\cG_t}\right) \rd s
     \end{align}
     By plugging \eqref{eq: linear KL2} into \eqref{def: robust control} and taking $h\downarrow 0$, we have $1 - \rho(s-t) \uparrow e^{-\rho(s-t)}$ and \eqref{def: robust control 2} satisfied as a consequence. 
\end{proof}



Before showing that \eqref{def: robust control 2} is a solution to some FBSDEs, let us introduce some norms and spaces, and the assumption to which such FBSDEs uniquely exist. 
To that end, we let $t\in [0,T]$ and denote by:
\begin{itemize}
    \item $\CL_T^{\infty}$ the set of the almost surely bounded $\cG_T-$measurable random variable $\xi$, i.e. $|\xi|<\infty$, $\Pp-$a.s. 
    \item $\CL^2_T$ the set of $\cG_T-$measurable random variable $\xi$ such that $\Ep^{\Pp}[|\xi|^2]<\infty$.
    \item $\CH^{\infty}_T$ the set of the almost surely bounded $\bG-$adapted real-valued process $X=\pro{X}$, i.e. $|X_t|<\infty$, $\Pp-$a.s.
    \item $\CH^{2}_T$ the set of $\bG-$adapted real-valued process $X=\pro{X}$ such that $ \Ep^{\Pp}_{\cG_t}\left[\int_t^{T}|X_s|^2 \rd s\right]< \infty$.
\end{itemize}

\begin{definition}\label{ast: terminal+driver}
Given that $t\in[0,T]$, $r=(x,\ldots,x_d)\in \sr^d$ and $(y_i,z_i)\in\sr^2$ for $i=1,2$. The random variable $\xi$ and mapping $F$ are said to be \textit{proper} if
    \begin{enumerate}[{1.)}]
        \item $\xi\in\CL^{\infty}_T$
        \item $F$ is continuous in $(t,r)$.
        \item There exists $C_1<\infty$ such that 
        \begin{align}
            |F(t,r,y_i,z_i)| &\leq C_1(1 + |y_i|+|z_i|^2)\\
            |F(t,r,y_1,z_1) - F(t,r,y_2,z_2)|&\leq C_1\Big( |y_1-y_2|+(1+|y_1|+|y_2|+|z_1|+|z_2|)|z_1-z_2|\Big).
        \end{align}
    \end{enumerate}
\end{definition}
We call a random variable $\xi$ and a mapping $F$, \emph{terminal value} and \emph{driver}, respectively.

The following lemma will later be essential for the verification theorem. The proof can be found in \citet{Zhang2017-jp}, and so is omitted here.
\begin{lemma}[Exponential Martingale]\label{lem: exp mart}
    Let $\delta\in \CH^2_T$ with norm
    \begin{align}
         \Ep^{\Pp}_{\cG_t}\left[\int_t^T |\delta_s|^2 \rd s \right] \leq C_2<\infty,
    \end{align}
    for any $t\in [0,T]$, and suppose that the process $\Gamma=\pro{\Gamma}$ solves 
    \begin{align}
        \rd \Gamma_t = \delta_t \Gamma_t \rd \overline B_t\;\;\;\Gamma_0=1. 
    \end{align}
    Then
    \begin{align}
         \Ep^{\Pp}_{\cG_t}\left[ \expo{\epsilon \int_t^T |\delta_s|^2 \rd s}  \right] \leq \frac{1}{1-\epsilon C_2}, 
    \end{align}
    for any $\epsilon\in(0, \frac{1}{C_2})$, and there is $\epsilon'\in(0, \frac{1}{C_2})$ such that
    \begin{align}
         \Ep^{\Pp}_{\cG_t}\left[ \sup_{t\in [0,T]} |\Gamma_t|^{1+\epsilon} \right] \leq C_{\epsilon'}, 
    \end{align}
    where $C_{\epsilon'}$ depends on $C_2$. In particular, $\Gamma$ is a uniformly integrable martingale.
\end{lemma}

The existence and uniqueness result for FBSDEs satisfying Assumption \ref{ast: terminal+driver} are given by the following comparison theorem. For further detail, we refer to \citet{Kobylanski2000-zw} or \citet{Zhang2017-jp}.

\begin{lemma}[Comparison Theorem]\label{thm: comparison}
    Suppose that the process $(X_i ,A_i, Y_i, Z_i)\in(\CH^2_T)^{d}\times(\CH^{\infty}_T)^{d+1}\times\CH^{\infty}_T\times(\CH^2_T)^{d+1},$ is a solution to 
    \begin{align}
        -\rd Y_{i,t} = F_i(t,X_i,Y_i,Z_i) \rd t +  \rd A_{i,t}  - Z_{i,t} \cdot \rd \boldsymbol{B}_t,\;\;\;Y_{i,T}=\xi_{i,T}
    \end{align}
    where $\boldsymbol{B}$ is a $(d+1)-$Brownian motions under  $\Pp$, $X_i$ satisfy \eqref{eq: suff con a} and \eqref{eq: suff con s}, $A_i\in \cD^{d+1}$, $\xi_{i}$ and $F_i$ are proper, for $i=1,2$. Suppose, furthermore, that:
    \begin{enumerate}[{1.)}]
        \item $A_1 - A_2 \geq \boldsymbol{0}$ a.s.,
        \item $\xi_{1} - \xi_{2} \geq 0$ a.s. and
        \item $F_1 - F_2 \geq 0$ a.e.,
    \end{enumerate}
    where denoted by $\boldsymbol{0}$ is the $(d+1)$-dimensional null-vector. Then $Y_{1,t} - Y_{2,t} \geq 0 $ a.s. for $t\in [0,T]$.
\end{lemma}

Now we are ready to establish the FBSDE represention for the value function \eqref{def: value function}.
\begin{theorem}\label{thm: FBSDEs}
    There exists a unique solution to $(R^{M^{\Pp},A},A,J^{A},Z) \in (\CH^2_T)^3\times(\CH^{\infty}_T)^2\times\CH^{\infty}_T\times\CH^2_T $, where $R^{M^{\Pp},A}\triangleq(X^{M^{\Pp},A},M^{\Pp},S)$, such that
    \begin{align}
     &\rd X^{M^{\Pp},A}_t &&= (\alpha(X^{M^{\Pp},A}_t) -\sigma(X^{M^{\Pp},A}_t)M^{\Pp}_t )\rd t + \rd A_t+\sigma(X^{M^{\Pp},A}_t)\rd \overline B_t,&&  X^{M^{\Qp},A}_0=x, \label{eq: X^A_t}\\
         &\rd M_t^{\Pp} &&= S_t\rd \overline B_t +\rd D_t , && M^{\Pp}_0=m, \label{eq: M_t}\\
         &\rd S_t &&= -S_t^2 \rd t, && S_0=s, \label{eq: S_t}\\
         &\rd J_t^{A}&&= -F(t,R^{M,A}_t,J_t^{A},Z_t) \rd t  -\ell \rd A^-_t - u \rd A^+_t + Z_t \rd \overline B_t, &&J^{A}_T= \xi(X^{M^{\Pp},A}_T), \label{eq: J^A_t}
    \end{align}
    where 
    \begin{align}
       &F(t,r,y,z)\triangleq F(t,x,m,s,y,z)\triangleq -\rho y + f(x) + \frac{\gamma s}{2}z^2, \\
       &f\;\text{is convex in } x,
    \end{align}
    for $(t,r,y,z)\in[0,T]\times \sr^5$.
    Moreover, $J^{A}$ satisfies the inequality \eqref{def: robust control 2}, where equality holds for which
    \begin{align}\label{def: overline Q}
         \frac{\rd \Pp^{\phantom{\ast}}}{\rd \Qp^\ast}\Big|^{(s,t+h)}_{\cG_t} = \frac{\expo{\gamma (\theta_s - M^{\Pp}_s )Z_s}}{ \Ep^{\Pp}_{\cG_s}\left[\expo{\gamma (\theta_s - M^{\Pp}_s )Z_s} \right]},
    \end{align}
    where $s\in[t,t+h)$.
\end{theorem}

\begin{proof}[Proof of Theorem \ref{thm: FBSDEs}]
    We first establish existence and uniqueness of $(R^{M,A},A,J^{A},Z)$. It is clear that $R^{M,A}$ admits a unique solution since $\alpha$ and $\sigma$ satisfy \eqref{eq: suff con a} and \eqref{eq: suff con s}, respectively. The Riccati equation \eqref{eq: S_t} gives that $S_t = \frac{s}{1+st},\;\; t\in[0,T]$, which implies that $\sup_{t\geq0}|S_t|\leq s$. 
    The continuity of $f$ in $R^{M,A}$ and its Lipschitz continuity with linear growth in $J^{A}$ ensure that
    \begin{align*}
        |F(t,r,y,z)|&\leq C_f |y|  + \frac{\gamma |s|}{2}z^2 \\
        |F(t,r,y_1,z_2) - F(t,r,y_2,z_2)| &\leq C_f |y_1-y_2| + \frac{\gamma |s|}{2}(z_1^2-z_2^2) \\&= C_f |y_1-y_2| + \frac{\gamma |s|}{2}(z_1+z_2)(z_1-z_2)
    \end{align*}
    where $C_f$ is the Lipschitz constant of $f$. Therefore, if we choose $C_1=$ $C_f + \frac{\gamma |s|}{2}$, then $F$ is proper. Together with the boundedness of $\xi$, we conclude that there is a unique solution to $(R^{M^{\Pp},A},A,J^{A},Z) \in (\CH^2_T)^3\times(\CH^{\infty}_T)^2\times\CH^{\infty}_T\times\CH^2_T $.
    
    Next, we show that $J^{A}_t$ satisfies the inequality \eqref{def: robust control 2}. Note that \eqref{eq: J^A_t} implies
    \begin{align}\label{eq: EY}
        J^{A}_t &=  \Ep^{\Pp}_{\cG_t}\Bigg[ \int_t^{t+h}   \left[ \left( -\rho J_s^{A} +f(X^{M^{\Pp},A}) +\frac{\gamma S_s}{2}Z^2_s  \right)\rd s +\ell \rd A^-_s + u \rd A^+_s\right] +  J^{A}_{t+h}  \Bigg], 
    \end{align}
    for some $h>0$ such that $t+h\leq T$. Now we fix $\Qp\in\cP$.
    Consequently, for all $s\in[t,t+h)$, one can infer that \begin{align}\label{eq: SZ}
        \frac{\gamma S_{s}}{2}Z^2_{s}&= \frac{1}{\gamma} \log\left( \expo{\gamma \Ep^{\Pp}_{\cG_s}[\theta_s - M^{\Pp}_s]Z_s + \frac{\gamma^2}{2} \Ep^{\Pp}_{\cG_s}[(\theta_s - M^{\Pp}_s)^2]Z^2_s} \right)\nonumber\\
        &= \frac{1}{\gamma}\log \Ep^{\Pp}_{\cG_s}\left[\expo{\gamma (\theta_{s}-M^{\Pp}_{s})Z_{s}}\right]\nonumber\\
        &\geq \Ep^{\Qp}_{\cG_s}\left[(\theta_s-M^{\Pp}_s)Z_s + \frac{1}{\gamma} \log\left( \frac{\rd \Pp}{\rd \Qp} \Big|^{(s,t+h)}_{\cG_s} \right)  \right] \nonumber\\
        &= - (M^{\Pp}_s -  M^{\Qp}_s)Z_s + \Ep^{\Qp}_{\cG_s}\left[\frac{1}{\gamma} \log\left( \frac{\rd \Pp}{\rd \Qp} \Big|^{(s,t+h)}_{\cG_s} \right)   \right].
    \end{align}
     The second equality follows from the moment generating function of $\theta_{s}-M^{\Pp}_{s}$, while the definition of logarithmic transformation (similar to Lemma \ref{lem: log-tran}) gives rise to inequality in the third line. Thus, we have
    \begin{align}\label{eq: EY2}
        J^{A}_t &\geq \Ep^{\Pp}_{\cG_t} \Bigg[ \int_t^{t+h}   \bigg[ \Bigg(  -\rho J_s^{A} +f(X^{M^{\Pp},A}) - (M^{\Pp}_s -  M^{\Qp}_s)Z_s \nonumber\\&\hspace{3cm}+
        \Ep^{\Qp}_{\cG_s}\left[\frac{1}{\gamma} \left( \log\frac{\rd \Pp}{\rd \Qp}\Big|^{(s,t+h)}_{\cG_s} \right) \right]  \Bigg)\rd s+\ell \rd A^-_s + u \rd A^+_s\bigg] +  J^{A}_{t+h} \Bigg] .
    \end{align}
   The change of measure from $\Pp$ to $\Qp$ gives
    \begin{align*}
        \overline B_t= \widehat B_t - \int_0^t (M^{\Qp}_s -  M^{\Pp}_s )\rd s,\;\;\;t\in[0,T],
    \end{align*}
    where $\widehat B$ is a $\Qp$--innovation process.
   As a result, we obtain that
    \begin{align*}
        J^{A}_t &\geq\Ep^{\Qp}_{\cG_t} \Bigg[  \int_t^{t+h}   \Bigg[ \left(  -\rho J_s^{A} +f(X^{M^{\Pp},A}) + \Ep^{\Qp}_{\cG_s}\left[\frac{1}{\gamma} \left( \log\frac{\rd \Pp}{\rd \Qp}\Big|^{(s,t+h)}_{\cG_s}\right)  \right] \right)\rd s \\&\hspace{8cm}  +\ell \rd A^-_s + u \rd A^+_s\Bigg]+  J^{A}_{t+h}\Bigg].
    \end{align*}
    Since $M^{\Pp}$ is càdlàg, taking $h\downarrow 0$ gives that $\cG_s\downarrow\cG_t$. Combining with the arbitrariness of $\Qp\in \cP$ and Lemma \ref{lem: log-tran}, we have \begin{align}\label{def: robust control 3}        
        J^{A}_t\geq&\sup_{\Qp\in\cP}\Ep^{\Qp}_{\cG_t}\Bigg[\int_t^{T} \bigg[\left(-\rho J_s^{A} + f(X^{M^{\Pp},A}) + \frac{1}{\gamma} \log\left(\frac{\rd \Pp}{\rd \Qp} (t,s) \right)\right)\rd s \nonumber\\&\hspace{7cm} +\ell \rd A^-_s + u \rd A^+_s \bigg] + \xi(X^{M^{\Pp},A}_T)\Bigg],
    \end{align}
    By It\={o}'s lemma, we simply have \eqref{def: robust control 2}.
    
    Finally, we choose $\Qp^\ast\sim\Pp$ such that the corresponding Radon-Nikodym derivative takes the form 
    \begin{align}\label{def: overline Q 2}
          \frac{\rd \Pp^{\phantom{\ast}}}{\rd \Qp^\ast}\Big|^{(s,t+h)}_{ \cG_{s}} = \frac{\expo{\gamma (\theta_s - M^{\Pp}_s )Z_s}}{ \Ep^{\Pp}_{\cG_s}\left[\expo{\gamma (\theta_s - M^{\Pp}_s )Z_s} \right]}.
    \end{align}
    Due to the dual formula for variational inference given by \citet{Lee2022-qb}, Theorem 2.1, and the integrability of all integrands standing in \eqref{def: overline Q}, it holds that the Radon-Nikodym derivatives~\eqref{def: overline Q 2} is well-defined. By plugging \eqref{def: overline Q} into \eqref{eq: SZ}, it is easy to see that equation \eqref{eq: SZ} achieves the equality. This infers that the equality holds for \eqref{def: robust control 3} and $ \Qp^\ast$ gives rise to the supremum of the robust control problem (in limit),
which is the completion of Theorem \ref{thm: FBSDEs}.
\end{proof}

\section{Heuristic derivation of the optimal control policy}\label{sec:derivation}
Before we develop a heuristic derivation of the optimal control policy, we first show that the value function \eqref{def: value function} satisfies the dynamic programming principle (DPP). 

\begin{lemma}[Dynamic Programming Principle]\label{lem: DPP}
    Suppose that for a given feasible policy $A\in \cD$, $J^{A}$ is the solution to \eqref{eq: J^A_t}. 
    Then for any $h\in[0,T]$ and $t\in[0,T-h]$ it holds that
    \begin{align}\label{eq: DPP}
        \inf_{A\in \cD} J^{A}_t = \inf_{A\in \cD}  \Ep^{\Pp}_{\cG_t}\Bigg[ \int_t^{t+h}   \left( F(s,R^{M,A}_s,J_s^{A},Z_s)  \rd s +\ell \rd A^-_s + u \rd A^+_s \right)  + \inf_{A\in \cD} J^{A}_{t+h}      \Bigg]
    \end{align}
\end{lemma}

\noindent\begin{proof}[Proof of Lemma \ref{lem: DPP}]
    By definition, we have
    \begin{align*}
        J^{A}_t &=  \Ep^{\Pp}_{\cG_t}\Bigg[ \int_t^{T}   \bigg( F(s,R^{M,A}_s,J_s^{A},Z_s)  \rd s +\ell \rd A^-_s + u \rd A^+_s \bigg) + \xi(X^{M^{\Pp},A}_T)     \Bigg] \\
        &=  \Ep^{\Pp}_{\cG_t}\Bigg[ \int_t^{t+h}   \bigg( F(s,R^{M,A}_s,J_s^{A},Z_s)  \rd s +\ell \rd A^-_s + u \rd A^+_s \bigg) +  J^{A}_{t+h}      \Bigg].
    \end{align*}
    Thus,
    \begin{align*}
        \inf_{A\in \cD} J^{A}_t &=\inf_{A\in \cD}   \Ep^{\Pp}_{\cG_t}\Bigg[ \int_t^{t+h}   \bigg( F(s,R^{M,A}_s,J_s^{A},Z_s)  \rd s +\ell \rd A^-_s + u \rd A^+_s \bigg) +  J^{A}_{t+h}      \Bigg]\\
        &\geq \inf_{A\in \cD}  \Ep^{\Pp}_{\cG_t}\Bigg[ \int_t^{t+h}   \bigg( F(s,R^{M,A}_s,J_s^{A},Z_s)  \rd s +\ell \rd A^-_s + u \rd A^+_s \bigg) + \inf_{A\in \cD} J^{A}_{t+h}      \Bigg].
    \end{align*}
    We next show that the opposite inequality also holds. Since $\inf_{A\in \cD} J^{A}_{t+h}$ is minimal among other feasible policies, the for any $\epsilon>0$, there exists a policy $\overline A\in \cD$ such that $\inf_{A\in \cD} J^{A}_{t+h} \geq  J^{\overline{L},\overline{U}}_{t+h} - \epsilon$. It follows that 
    \begin{align*}
        \inf_{A\in \cD} J^{A}_t &\leq   \Ep^{\Pp}_{\cG_t}\Bigg[ \int_t^{t+h}   \left( F(s,R^{M,\overline A}_s,J_s^{\overline A},Z_s)  \rd s +\ell \rd \overline A^-_s + u \rd \overline A^+_s \right) +  J^{\overline{L},\overline{U}}_{t+h}      \Bigg]\\
        &\leq   \Ep^{\Pp}_{\cG_t}\Bigg[ \int_t^{t+h}   \left( F(s,R^{M,\overline A}_s,J_s^{\overline A},Z_s)  \rd s +\ell \rd \overline A^-_s + u \rd \overline A^+_s \right) +  \inf_{A\in \cD} J^{A}_{t+h} + \epsilon      \Bigg]
    \end{align*}
    By arbitrariness of $\epsilon$, Lemma \ref{lem: DPP} is obtained.
\end{proof}

In order to solve \eqref{eq: DPP}, we first try the classical approach of the optimal (singular) control by \emph{assuming} that there is a $C^{1,2}([0,T]\times \sr^3)$ function $v(t,r)$ such that $v(t,r) = \inf_{A\in \cD} J^{A}_t$ given that $R^{M,A}_{t} = r= (x,m,s)$. If this assumption holds, then we can rewrite \eqref{eq: DPP} as
\begin{align}\label{eq: DPP v}
    v(t,r) &= \inf_{A\in \cD}  \Ep^{\Pp}_{\cG_t}\Bigg[ \int_t^{t+h}   \bigg( F(s,R^{M,A}_s,J_s^{A},Z_s)  \rd s +\ell \rd A^-_s + u \rd A^+_s \bigg)  + v(t+h,R^{M,A}_{t+h})      \Bigg]
\end{align}
By applying It\={o}'s lemma to $v(t+h,R^{M,A}_{t+h})$ we can see that
\begin{align}\label{eq: HJB1}
    v(t,r) &= \inf_{A\in \cD}  \Ep^{\Pp}_{\cG_t}\Bigg[ v(t,r) + \int_t^{t+h}   \Bigg[  \bigg(\cL v(s,R^{M,A}_s)  +F(s,R^{M,A}_s,J_s^{A},Z_s) \bigg) \rd s \nonumber \\ & \hspace{4cm} + \left(\ell +\diff{v(s,R^{M,A}_s)}{x}\right)\rd A^-_s + \left(u-\diff{v(s,R^{M,A}_s)}{x}\right)\rd A^+_s   \Bigg]  \nonumber\\ & \hspace{4cm} + \diff{v(s,R^{M,A}_s)}{m}\rd D^-_s  -\diff{v(s,R^{M,A}_s)}{m}\rd D^+_s   \Bigg]      \Bigg]
\end{align}
where $\cL$ is so-called the infinitestimal generator operator of $v$, and defined by
\begin{align*}
    \cL v \triangleq \diff{v}{t}  + [\alpha(x)-m\sigma(x)]&\diff{v}{x} -s^2\diff{v}{s}  +\frac{\sigma^2(x)}{2}\frac{\partial^2v}{\partial x^2} +\frac{s^2}{2}\frac{\partial^2v}{\partial m^2} + s\sigma(x)\frac{\partial^2v}{\partial x\partial m}
\end{align*}
It is important to note that in order for a feasible policy to be optimal, it is sufficient from \eqref{eq: HJB1} that for any $A\in \cD$, 
\begin{align}\label{def: actions inequalities}
    \cL v(s,r) + F(t,r,y,z) \geq 0,\;\;\;\;  \ell +\diff{v(s,r)}{x} \geq 0,\;\;\;\;\textup{and}\;\;\;\;
     u - \diff{v(s,r)}{x} \geq 0,
\end{align}
for any $(s,r)\in [0,T]\times \sr^3$. That is to say,
\begin{align*}
    0 &\leq  \Ep^{\Pp}_{\cG_t}\Bigg[ \int_t^{t+h}   \Bigg[  \left(\cL v(s,R^{M^{\Pp},A}_s) +  F(s,R^{M^{\Pp},A}_s,J_s^{ A},Z_s) \right) \rd s \nonumber \\ & \hspace{4cm} + \left(\ell +\diff{v(s,R^{M^{\Pp},A}_s)}{x}\right)\rd  L^x_s + \left(u-\diff{v(s,R^{M^{\Pp},A}_s)}{x}\right)\rd  U^x_s   \Bigg]  \nonumber\\ & \hspace{4cm} + \diff{v(s,R^{M,A}_s)}{m}\rd D^-_s  -\diff{v(s,R^{M,A}_s)}{m}\rd D^+_s   \Bigg]      \Bigg],
\end{align*}
and the equality holds if there exists the optimal control policy $A^\ast\in \cD$. We show later by the verification theorem that these are also necessary conditions. Otherwise, a contradiction occurs. 

In order to establish the optimality conditions for the feasible control, let us recall the main idea behind  classical singular control given by \citet{harrison1983instantaneous}. Now we define by
\begin{align}
     h_{\underline m} \triangleq \inf\left\{h>0 :  M^{\Pp}_{t+h}=\mL\right\}\text{  and    }h_{\overline m} \triangleq\inf\left\{h>0 :  M^{\Pp}_{t+h}=\mU\right\}
\end{align} 
a $\cG_t$--stopping time. Here the operator $\wedge$ is that $a\wedge b = \min\{a,b\}$ for any $a,b\in\sr$. This stopping time determines a sequence of actions for the DM to consider. That is, for a given $\epsilon>0$ if $h_{\underline m}\wedge h_{\overline m}<\epsilon$, DM must reflect the belief process upward or downward, respectively, to keep $M^{\Pp}\in [\mL,\mU]$. Therefore, the value function is either
\begin{align}
    v_{\mL}(t,r)  &\triangleq \Ep^{\Pp}_{\cG_t}\left[v(t,r)+\int_t^{t+h_{\underline m}}\diff{v(s,R^{M^{\Pp},A}_s)}{m}\rd   D^-_s \right],\;\;\textup{or} \\
    v_{\mU}(t,r)  &\triangleq\Ep^{\Pp}_{\cG_t}\left[v(t,r) - \int_t^{t+h_{\overline m}}\diff{v(s,R^{M^{\Pp},A}_s)}{m}\rd   D^+_s \right].
\end{align}
On the other hand, if $h_{\underline m}\wedge h_{\overline m} > \epsilon$, there are three options to exercise;
(i) to exert the lower or (ii) upper singular control, or (iii) to remain inactive and incur the continuation value. It follows from \eqref{eq: HJB1} that the value function on $[t,t+\epsilon)$ for these options is either
\begin{align}
    v_{\ell}(t,r)  &\triangleq \Ep^{\Pp}_{\cG_t}\left[v(t,r)+\int_t^{t+\epsilon}\left(\ell +\diff{v(s,R^{M^{\Pp},A}_s)}{x}\right)\rd    A^-_s \right],\\
    v_{u}(t,r)  &\triangleq \Ep^{\Pp}_{\cG_t}\left[v(t,r)+\int_t^{t+\epsilon}\left(u -\diff{v(s,R^{M^{\Pp},A}_s)}{x}\right)\rd    A^+_s \right], \;\;\textup{or} \\
     v_c(t,r)  &\triangleq \Ep^{\Pp}_{\cG_t}\left[v(t,r)+\int_t^{t+\epsilon}\left(\cL v(s,R^{A}_s)  +F(s,R^{M^{\Pp},A}_s,J_s^{  A},Z_s)\right)\rd s \right],
\end{align}
respectively. Here there DM must take the option that minimises the cost function, then value function at $t+h_m\wedge\epsilon$ satisfies
\begin{align}\label{eq: HJB2}
    v(t,r)=\mathds1_{\{h_{\underline m}\wedge h_{\overline m}>\epsilon\}}\min\left\{v_{\ell}(t,r),v_{u}(t,r),v_{c}(t,r)\right\} +  \mathds1_{\{h_{\underline m}< h_{\overline m}<\epsilon\}}v_{\mL}(t,r)+  \mathds1_{\{h_{\overline m}< h_{\underline m}<\epsilon\}}v_{\mU}(t,r).
\end{align}
Dividing \eqref{eq: HJB2} by $h_{\underline m}\wedge h_{\overline m}\wedge\epsilon$ and send $\epsilon \to 0$, the MVT gives
\begin{align}
    0&=\min\left\{\cL v(t,r) +F(s,r,J_t^{A^\ast},Z_t), \;\;\ell + \diff{v(t,r)}{x},\;\; u - \diff{v(t,r)}{x}  \right\}, \\
    0&=\diff{v(t,x,\mL+,s)}{m}=\diff{v(t,x,\mU-,s)}{m}\\
    v(T,r)&=\xi(x).
\end{align}
By It\={o}'s lemma and the definition of $v$ it follows that
\begin{align}\label{eq: sol1}
    J_t^{A^\ast} = v(t,r)\;\;\textup{and}\;\;Z_t = \sigma(x) \diff{v(t,r)}{x} + s \diff{v(t,r)}{m}
\end{align}
Therefore, if the regularity of $v$ holds then we finally have that $v$ is the solution to the following HJB equation for singular control,
\begin{align}\label{eq: HJB3}
    0&=\min\left\{\cL v(t,r) + F\left(t,r,v(t,r),\sigma(x) \diff{v(t,r)}{x} + s \diff{v(t,r)}{m}\right),\;\; \ell + \diff{v(t,r)}{x},\;\; u - \diff{v(t,r)}{x}  \right\},\\
    0&=\diff{v(t,x,\mL+,s)}{m}=\diff{v(t,x,\mU-,s)}{m}\\
    v(T,r)&=\xi(x).
\end{align}
Now we defines by $\CC_v,\LB_v,\UB_v\subset[0,T]\times\sr$ such that
\begin{align}
    \LB_v&\triangleq\left\{(t,x)\in [0,T]\times\sr: -\ell\geq\diff{v(t,r)}{x}\right\}\label{def: lower stopped region}\\
    \UB_v&\triangleq\left\{(t,x)\in [0,T]\times\sr: \phantom{-}u\leq\diff{v(t,r)}{x} \right\}\;\;\;\text{and,}\label{def: upper stopped region}\\
    \CC_{ v}&\triangleq[0,T]\times\sr\setminus (\LB_{v}\cup\UB_{ v}),\label{def: continuation region}
\end{align}
It can be noticed as regard the definition of $\CC_v,\LB_v$ and $\UB_v$ that HJB equation \eqref{eq: HJB3} is equivalent to solving
\begin{align}
    \cL v(t,r) + F\left(t,r,v(t,r),\sigma(x) \diff{v(t,r)}{x} + s \diff{v(t,r)}{m}\right)&=0\;\;\;(t,x)\in\CC_v \label{eq: continue PDE}\\
    \ell + \diff{v(t,r)}{x}&=0\;\;\;(t,x)\in\LB_v\label{eq: lower barrier PDE}\\
    u - \diff{v(t,r)}{x}&=0\;\;\;(t,x)\in\UB_v.\label{eq: upper barrier PDE}\\
    \diff{v(t,x,\mL+,s)}{m}=\diff{v(t,x,\mU-,s)}{m}&=0 \label{eq: m barrier PDE}
\end{align}
This tells us that it is optimal to incur the continuation value, rather than exercising control, as long as the inventory stock explores in $\CC_v$. In other words, the control policy $A$ do not increase in this region. Mathematically speaking, it must holds that
\begin{align}\label{ast:not incressing L U}
    \int_0^T \mathds{1}_{\left\{(t,X^{M^{\Pp},A}_t)\in \CC_v\right\}} \rd (L^x_t + U^x_t)=0,\; \Pp-\textup{a.s.}
\end{align}
in which to be shown to be crucial in the verification theorem for the existence and uniqueness of the optimal control policy. On the other hands, once the inventory reaches $\LB_v$, for the HJB equation \eqref{eq: HJB3} to holds, the lower control must be exerted to keep the inventory within the continuation region. The upper control is taken into account in the same way  whenever $(t,R^{A})\in \UB_v$. 
For these reasons, we refer to the space $\CC_v,\LB_v$ and $\UB_v$ as a \textit{continuation region}, \textit{lower control barrier} and \textit{upper control barrier}, respectively. Moreover, we realise by construction of $F$ 
 that $v$ is \textit{convex} in $x$ (cf. \citet{Wang2004-js}, Proposition 3), which ensures that $\CC_v\cap\LB_v\cap\UB_v=\varnothing$. In other words, at most one decision is made, either to do nothing (when $(t,x)\in\CC_v$), or else to take either upper or lower controls.

However, the earlier claims are invalid if the value function is non-smooth, in which case there is no classical solution to \eqref{eq: HJB3}. We show in the next section that only continuity is required to construct the value function as a viscosity solution to the HJB equation.


\section{Viscosity Solution of the Value Function}\label{sec:viscosity}
Given the previous argument, we cannot immediately say that the BSDE $(J^{A^\ast}, Z)$ given by \eqref{eq: sol1} is the solution the fully nonlinear PDE \eqref{eq: HJB3} (at least in a classical sense) since $J^{A^\ast}$ is only known to be continuous, but not necessary smooth, opposing to the assumption of $v$.  To deal with is issue, we employ the notion of viscosity solution introduced by \citet{crandall1983viscosity}, in which they generalise the solution of fully nonlinear PDEs to equip in a weaker sense, where the only locally boundedness is assumed. The development of a viscosity solution to the second order PDEs thanks to \citet{crandall1987uniqueness}, and our following argument will rely on this paper. In other words, we show that $J^{A^\ast}$ is a viscosity solution to \eqref{eq: HJB3}. Before doing so, let us first state the definition of viscosity solution.
\begin{definition}[Viscosity Solution]
    For all $v(t,r)\in C([0,T]\times \sr^3)$ and $\phi\in C^{1,2}([0,T]\times \sr^3)$, define the operator $\mathbb{L}$, for any $(t,r)\in[0,T]\times \sr^3$, by 
    \begin{equation}\label{eq:viscosity}
        \mathbb{L}\phi(t,r) \triangleq \cL\phi(t,r) + F\left(t,r,\phi(t,r),\sigma(x) \diff{\phi(t,r)}{x} + s \diff{\phi(t,r)}{m}\right).
    \end{equation}
    \begin{enumerate}[{1.)}]
        \item The function $v$ is a viscosity subsolution of~\eqref{eq: HJB3} if for all local maximum points $(t',r')\in[0,T]\times \sr^3$ of $v-\phi$ it holds that
        \begin{align}
            \min\Bigg\{ \mathbb{L}\phi(t',r'),\;\; \ell + \diff{\phi(t',r')}{x},\;\; u - \diff{\phi(t',r')}{x}  \Bigg\}\geq 0, \hspace{1cm} \phi(T,r')=h(x'). \label{eq: subsolution}
        \end{align}
        
        \item The function $v$ is a viscosity supersolution of~\eqref{eq: HJB3} if for all local minimum points $(t',r')\in[0,T]\times \sr^3$ of $v-\phi$ it holds that
        \begin{align}
            \min\Bigg\{ \mathbb{L}\phi(t',r'),\;\; \ell + \diff{\phi(t',r')}{x},\;\; u - \diff{\phi(t',r')}{x}  \Bigg\}\leq 0, \hspace{1cm} \phi(T,r')=h(x'), \label{eq: supersolution}
        \end{align}
        
        \item $v(t,r)\in C([0,T]\times \sr^3)$ is a viscosity solution of~\eqref{eq: HJB3} if it both viscosity subsolution and supersolution.
    \end{enumerate}
\end{definition}

\begin{remark}[]
    Given the possible non-smoothness of $u$, the construction of inaction or action regions, as outlined in definitions~\eqref{def: continuation region} to \eqref{def: upper stopped region}, becomes impractical. This also means that the condition~\eqref{ast:not incressing L U} is invalid. Consequently, it is imperative to revise~\eqref{def: continuation region} to \eqref{def: upper stopped region} and strengthen assumption~\eqref{ast:not incressing L U} to accommodate similar considerations for the non-smoothness of $u$. These adjustments will be shown in the verification theorem, highlighting their necessity for establishing the existence and uniqueness of the viscosity solution. 
\end{remark}

\begin{definition}
    For given $v\in C[0,T]\times \sr^3$ and $h>0$,
    \begin{align}
    \LB_v^h&\triangleq\left\{(t,x)\in [0,T]\times\sr: -\ell h\geq v(t,x,m,s)-v(t,x-h,m,s)\right\}\label{def: refined lower stopped region}\\
    \UB_v^h&\triangleq\left\{(t,x)\in [0,T]\times\sr: \phantom{-}uh\leq v(t,x+h,m,s)-v(t,x,m,s)\right\}\label{def: refined upper stopped region}\;\;\;\text{and,}\\
    \CC_{ v}^{h}&\triangleq[0,T]\times\sr\setminus (\LB_{v}^{h}\cup\UB_{v}^{h}),\label{def: refined continuation region}
\end{align}
\end{definition}

\begin{assumption}\label{ast:refined not incressing L U}
    Given  $v\in C[0,T]\times \sr^3$ and. For any feasible policy $A\in \cD$, there exists $h>0$ such that
    \begin{align}
        \int_0^T\mathds{1}_{\{(t,X^{M^{\Pp},A}_t)\in \CC_v^h\}} \rd (L^x_t + U^x_t) = 0,\;\Pp-\textup{a.s.}
    \end{align}
\end{assumption}

Now, we show that $J^{A^\ast}$ is indeed a viscosity solution of \eqref{eq: HJB3}.
\begin{theorem}\label{thm: viscosity 1}
    The value function defined by  \eqref{def: value function} is the unique continuous viscosity solution to \eqref{eq: HJB3}.
\end{theorem}

\noindent\begin{proof}[Proof of Theorem \ref{thm: viscosity 1}]\phantom{==}\\
    \textbf{Uniqueness: } First we claim that there exists a viscosity solution to \eqref{eq: HJB3}, which will be proved in the subsequent section. To this end, we suppose that $v$ is a viscosity solution to a smooth function $\phi\in C^{1,2}([0,T]\times \sr^3)$. Given $h>0$, it is straightforward from \citet{crandall1983viscosity} that $v$ admits a unique viscosity solution on the regions $\LB^h_v$ and $\UB^h_v$ since it is just a first order PDE.  Because $v$ is a viscosity solution to a PDE with quadratic in gradients on $\CC^h_v$, we have that there is also a unique viscosity solution on this region according to the argument given by \citet{Kobylanski2000-zw}. By smoothness of $\phi$ we conclude that $v$ has a unique viscosity solution.
    
    \textbf{Existence: }For what follows we adopt the idea from \citet{Zhang2017-jp}, which establishes the verification theorem for the non-controlled BSDE with a regular driver i.e. $f$ is Lipchitz in $Z$. Then we develop the argument to equip the theorem with singular controls and a driver with quadratic growth in $Z$. 
    
    Suppose that $\phi\in C^{1,2}([0,T]\times \sr^3)$. It is obvious from the definition that at $T$, $v(T,r')=h(x')=\phi(T,r')$. Thus, it is left to show the viscosity property of $v$ on $[0,T)\times \sr^3$. First, we prove that $v$ is a viscosity subsolution before proceeding to verify its supersolution property.

    \textbf{Subsolution:} we suppose $v-\phi$ has a maximum point at $(t',r')$. In other word, $(v-\phi)(t',r')=0$ and  $(v-\phi)(t',r')\geq (v-\phi)(t,r)$ for all $(t,r)\in [0,T)\times \sr^3$. Then we show that \eqref{eq: subsolution} holds. Suppose that $v(t,r)\triangleq \inf_{A\in \cD} J^{A}_{t}\triangleq J^{A^\ast}_{t}$. By DPP (Lemma \ref{lem: DPP}), it gives
    \begin{align}\label{eq: DPP v t'}
        v(t',r') = \Ep^{\Pp}_{\cG_{t'}}\left[ \int_{t'}^{t'+h}   \left( F\left(t,R_s^{A^\ast},v(s,R^{A^\ast}_{s}),Z_s\right)\rd s +\ell \rd A^{\ast-}_s + u \rd A^{\ast+}_s \right ) + v(t'+h,R^{A^\ast}_{t'+h})\right]
    \end{align}
    for $h>0$ such that $t'\leq t'+h \leq T.$ By It\={o}'s lemma, $\phi(t',r')$ solves
    \begin{align}
        \phi(t',r')=  \Ep^{\Pp}_{\cG_{t'}} \Bigg[-\int_{t'}^{t'+h}   \Bigg( \cL \phi(s,R^{A^\ast}_s) \rd s &+ \diff{\phi}{x}(s,R^{A^\ast}_s)\rd A^\ast_s +\diff{\phi}{m}(s,R^{A^\ast}_s)\rd D_s\Bigg) + \phi(t'+h,R^{A^\ast}_{t'+h})
        \Bigg]\label{eq: DPP phi t'}
    \end{align}
    Because $v-\phi$ has a maximum point at $(t',r')$, \eqref{eq: DPP v t'} and \eqref{eq: DPP phi t'} implies that
    \begin{align}
        0\leq  \Ep^{\Pp}_{\cG_{t'}}\Bigg[\int_{t'}^{t'+h}   \Bigg( &\left(\cL \phi(s,R^{A^\ast}_s) +  F\left(t,R_s^{A^\ast},v(s,R^{A^\ast}_{s}),Z_s\right) \right)\rd s \nonumber\\ &+ \left(\ell + \diff{\phi}{x}(s,R^{A^\ast}_s) 
        \right)\rd A^{\ast-}_s + \left(u - \diff{\phi}{x}(s,R^{A^\ast}_s) 
        \right)\rd A^{\ast+}_s +\diff{\phi}{m}(s,R^{A^\ast}_s)\rd D_s\Bigg)\Bigg]
     \end{align}
    Now we show that \eqref{eq: subsolution} holds by splitting the proof into two cases: when (i) for each $\omega_t\in\Omega$ such that $A^\ast(\omega_t)=(0,0)$ and (ii) such that $A^\ast(\omega_t)\neq(0,0)$, for $t\in[t',t'+h]$.
    
    \textbf{(i): }First of all, we choose $Z\triangleq\sigma(x)\diff{\phi}{x}+s\diff{\phi}{m}$. The smoothness of $\phi$ gives that $Z\in \CH^2_T$, implying that $v$ still admits a unique solution according to Theorem~\ref{thm: comparison}. Since the first case means there is no intervention on $[t',t'+h]$ and is one of the feasible policies as follows \eqref{cond:never intervene}, the integrals with respect to  $A^\ast$ vanish. Together with the MVT, these imply that $\mathbb{L}\phi(t',r')\geq 0$.

    \textbf{(ii): }Given that $A^\ast\neq(0,0)$, $\Pp-$a.s. on $[t',t'+h]$. We show that if $\ell + \diff{\phi}{x}(t',r')<0$ or $u - \diff{\phi}{x}(t',r')<0$ then they lead to some contradiction. By the fundamental calculus, we know that there exists $h>0$ such that 
    \begin{align*}
        \diff{\phi}{x}(t',r')\cdot h=\phi(t',x',m',s') - \phi(t',x'-h,m',s')
        =\phi(t',x'+h,m',s') - \phi(t',x',m',s')
    \end{align*}
    Since $v-\phi$ has a maximum point at $(t',r')$, it is easy to check that
    \begin{align*}
        h\ell +v(t',x',m',s') - v(t',x'-h,m',s')&\geq h\ell+ \phi(t',x',m',s') - \phi(t',x'-h,m',s')>0, \textup{ or}\\
        hu- v(t',x'+h,m',s') + v(t',x',m',s')&\geq hu-\phi(t',x'+h,m',s') + \phi(t',x',m',s')>0.
    \end{align*}
    This means that $(t', x') \notin \LB^h_{v}\cup\UB^h_{v}$. Since $v$ is convex in $x$, we arrive with a conclusion that   $(t', r') \in \CC^h_{v}$. By Assumption~\ref{ast:refined not incressing L U}, we then have $A^\ast=(0,0)$, $\Pp-$a.s. which is a contradiction. Therefore, this implies that $\ell+\diff{\phi}{x}(t',r')\geq0$ and $u - \diff{\phi}{x}(t',r')\geq 0$.
    
    All in all, we conclude that \eqref{eq: subsolution} holds and therefore $v$ is a viscosity subsolution.
    
    \textbf{Supersolution:} Now we show that $v$ is a viscosity supersolution. To this end, we suppose $v-\phi$ has a minimum point at $(t',r')$. In other word, it holds that $(v-\phi)(t',r')=0$ and  $(v-\phi)(t',r')\leq (v-\phi)(t,r)$ for all $(t,r)\in [0,T)\times \sr^3$. From now on, we proceed the argument by contradiction, which means in opposite to \eqref{eq: supersolution} that
    \begin{align}
        \mathbb{L}\phi(t',r')>0,\;\;\;
        \ell + \diff{\phi(t',r')}{x} >0 \;\;\;\textup{and}\;\;\;
        u - \diff{\phi(t',r')}{x} >0
    \end{align}
    At this point we may assume that there exist $C_c,C_{\ell},C_u>0$ such that
    \begin{align}
        \mathbb{L}\phi(t',r')=C_c,\;\;\;
        \ell + \diff{\phi(t',r')}{x} =C_{\ell} \;\;\;\textup{and}\;\;\;
        u - \diff{\phi(t',r')}{x} =C_u.\label{eq: sub contra}
    \end{align}
    Let $C'=\frac{C_c\wedge C_{\ell} \wedge C_u}{3}$, then by smoothness of $\phi$, the following the stopping time is well defined:
    \begin{align}\label{def: tau}
        \tau(t')\triangleq T \wedge\inf\Bigg\{\tau\geq t': &\Bigg(\mathbb{L}\phi(\tau,R_{\tau}^{A^\ast})\Bigg) \wedge  \left( \ell + \diff{\phi(\tau,R_{\tau}^{A^\ast})}{x} \right)\wedge\nonumber\\&\left( u - \diff{\phi(\tau,R_{\tau}^{A^\ast})}{x} \right) \wedge \left(\diff{\phi(\tau,R_{\tau}^{A^\ast})}{m} \right) \wedge \left(- \diff{\phi(\tau,R_{\tau}^{A^\ast})}{m}\right)\leq C'\Bigg\},
    \end{align}
     Thus, for $t\in[t', \tau(t')]$ we have by the DPP (Lemma \ref{lem: DPP}) that
    \begin{align}\label{eq: DPP2}
        J^{A^\ast}_{t'} =  \int_{t'}^{\tau(t')}   \left[ F\left(t,R_s^{A^\ast},J_s^{A^\ast},Z_s\right)\rd s +\ell \rd A^{\ast-}_s + u \rd A^{\ast+}_s - Z_s \rd \overline B_s \right] + J^{A^\ast}_{\tau(t')}
    \end{align}
    For each $t'\leq t \leq \tau(t')$. Now if we define by 
    \begin{align}
        \check J^{A^\ast}_{t'} = \phi(t',R^{A^\ast}_{t'}) \;\;\;\textup{and}\;\;\; \check Z_{t'} =  \sigma( X^{A^\ast}_{t'})  \diff{\phi(t,R^{A^\ast}_{t'})}{x} + S_{t'} \diff{\phi(t',R^{A^\ast}_{t'})}{m},
    \end{align}
    it then follows from the It\={o}'s lemma and \eqref{def: tau} that
    \begin{align}\label{eq: DPP4}
        \check J^{A^\ast}_{t'} &= \phi(\tau(t'),R^{A^\ast}_{\tau(t')}) - \int_{t'}^{\tau(t')}   \Bigg[ \cL \phi(s,R^{A^\ast}_s) \rd s + \diff{\phi}{x}(s,R^{A^\ast}_s)\rd A^\ast_s  \nonumber\\& \hspace{6.5cm} + \diff{\phi}{m}(s,R^{A^\ast}_s)\rd D_s + \check Z_s \rd \overline B_s \Bigg]\nonumber\\
        &\leq \phi(\tau(t'),R^{A^\ast}_{\tau(t')}) - \int_{t'}^{\tau(t')}   \bigg[ \left(C'-F\left(s,{R}_s^{A^\ast},\check{J}_s^{A^\ast},\check Z_s\right)\right)\rd s +  \check Z_s \rd \overline B_s\nonumber\\ &\hspace{5cm} + (C'-\ell)\rd A^{\ast-}_s + (C'-u)\rd A^{\ast+}_s  + C'(\rd D^-_s +\rd D^+_s)\bigg].
    \end{align}
    Since $(v-\phi)(t',r')=0$, it follows that the terms $\Delta J_t \triangleq J^{A^\ast}_t - \check J^{A^\ast}_t$ and $\Delta  Z_t \triangleq  Z_t - \check Z_t$ satisfy
    \begin{align*}
        -\Delta J_{\tau(t')} &\geq \int_{t'}^{\tau(t')}\Bigg[ \left(C'  + F(s,{R}_s^{A^\ast},{J}_s^{A^\ast}, Z_s)  -F(s,{R}_s^{A^\ast},\check{J}_s^{A^\ast},\check Z_s) \right)\rd s \\ &\hspace{2cm}+ C'(\rd A^{\ast-}_s+\rd A^{\ast+}_s + \rd L^{m}_s+\rd U^{m}_s) +  \left( \check Z_s -   Z_s \right)\rd \overline B_s  \Bigg]\\
        &\geq  \int_{t'}^{\tau(t')}\Bigg[ C'(\rd s + \rd A^{\ast-}_s + \rd A^{\ast+}_s + \rd L^{m}_s+\rd U^{m}_s) + 
        \left( \delta_s^J \Delta {J}_s + \delta_s^Z \Delta  Z_s  \right) \rd s -\Delta Z_s\rd \overline B_s\Bigg] \\
        &> \int_{t'}^{\tau(t')}\Bigg[
        \left( \delta_s^J \Delta {J}_s + \delta_s^Z \Delta  Z_s  \right) \rd s -\Delta Z_s\rd \overline B_s\Bigg],
    \end{align*}
    where
    \begin{align}
        \delta_s^J &\triangleq \frac{F(s, {R}_s^{A^\ast}, {J}_s^{A^\ast},\check Z_s) - F(s, {R}_s^{A^\ast},\check{J}_s^{A^\ast},\check Z_s)}{\Delta {Y}_s}\mathds{1}_{\Delta {J}_s\neq 0}\\
        \delta_s^Z &\triangleq \frac{F(s, {R}_s^{A^\ast}, {J}_s^{A^\ast},  Z_s) - F(s, {R}_s^{A^\ast}, {J}_s^{A^\ast},\check Z_s)}{\Delta {Z}_s}\mathds{1}_{\Delta {Z}_s\neq 0}.
    \end{align}
    The second inequality holds by \eqref{def: tau}. From the facts that $C'>0$ and that $A^\ast$ and $D$ are non-decreasing, we infer the last strict inequality. 
    
    From the comparison theorem one can see that
    \begin{align}\label{eq: delta}
        \delta_s^J \leq C_1 < \infty \;\;\textup{and}\;\;\delta_s^Z \leq C_1(1 + | {Y}_s^{A^\ast}|+|\check{J}_s^{A^\ast}|+| {Z}_s|+|\check{Z}_s|). 
    \end{align}
    Now we define the processes $(\Gamma^J,\Gamma^Z)\triangleq(\Gamma^J_t,\Gamma^Z_t)_{t\geq 0}$ by
    \begin{align}
        \rd \Gamma^J_t = \delta_t^y \Gamma^J_t \rd t,\;\;\Gamma_0^y=1,\;\;\textup{and}\;\;\rd \Gamma^Z_t = \delta_t^z \Gamma^Z_t \rd \overline B_t,\;\;\Gamma_0^z=1,
    \end{align}
    respectively. By It\={o}'s integration by parts, we have
    \begin{align*}
        \rd \Gamma^J_t \Delta J^{A}_t &=  \Gamma^J_t\rd \Delta J^{A}_t + \delta^y_t \Delta J^{A}_t \Gamma^J_t \rd t,
    \end{align*}
    and therefore
    \begin{align}\label{eq: gamma inequality}
        \rd \Gamma^Z_t\Gamma^J_t\Delta J^{A}_t &= \Gamma_t^z(\Gamma^J_t\rd \Delta J^{A}_t + \delta^y_t \Delta J^{A}_t \Gamma^J_t \rd t) +\delta_t^z \Gamma^Z_t\Gamma^J_t \Delta J^{A}_t \rd \overline B_t + \delta_t^z\Gamma_t^y\Gamma_t^z\Delta Z_t \rd t \nonumber\\
        &<   \Gamma_t^z(\Gamma^J_t
         [-\left(\delta_t^y \Delta {Y}_t+ \delta_t^z \Delta  Z_t  \right) \rd t  +\Delta Z_t\rd \overline B_t] + \delta^y_t \Delta J^{A}_t \Gamma^J_t \rd t)\nonumber \\& \hspace{6cm}+\delta_t^z \Gamma^Z_t\Gamma^J_t \Delta J^{A}_t \rd \overline B_t+ \delta_t^z\Gamma_t^y\Gamma_t^z\Delta Z_t \rd t\nonumber\\
        &=\Gamma^Z_t\Gamma^J_t(\Delta Z_t + \delta^z_t\Delta J^{A}_t)\rd \overline B_t.
    \end{align}
    From \eqref{eq: delta}, it is clear that $\delta^Z\in \CH^2_T$. Thus, we immediately obtain by Lemma \ref{lem: exp mart} that $\Gamma^Z$ is a uniformly integrable martingale which therefore induces that
    \begin{align}
        \Ep^{\Pp}_{\cG_{t'}}\left[ \int_{t'}^{\tau(t')} \Gamma^Z_s\Gamma^J_s(\Delta Z_s + \delta^z_s\Delta J_s)\rd \overline B_s\right] = 0.
    \end{align}
    This implies from \eqref{eq: gamma inequality} that
    \begin{align*}
        \Ep^{\Pp}_{\cG_{t'}}\left[ \Gamma^Z_{\tau(t')}\Gamma^J_{\tau(t')}\Delta J_{\tau(t')} -  \Gamma^Z_{t'}\Gamma^J_{t'}\Delta J_{t'}\right] <0
    \end{align*}
    However, since $v-\phi$ attains its minimum at $(t',r')$, it follows that $\Delta J_{\tau(t')}\geq\Delta J_{t'} = 0$. Therefore,
    \begin{align*}
        \Ep^{\Pp}_{\cG_{t'}}\left[ \Gamma^Z_{\tau(t')}\Gamma^J_{\tau(t')}\Delta J_{\tau(t')} -  \Gamma^Z_{t'}\Gamma^J_{t'}\Delta J_{t'}\right] \geq 0,
    \end{align*}
    which is a contradiction. As a result, the value function defined by  \eqref{def: value function} is a continuous viscosity supersolution to \eqref{eq: HJB3}.
    
    Hence, we conclude that $v$ is a viscosity solution to \eqref{eq: HJB3}.
\end{proof}

\section{Coordination transform and free boundaries characterisation}\label{sec:coor tran}

To solve the HJB equation \eqref{eq: HJB3}, we need to know the initial condition (IC) and the free boundary conditions (BC) for $(X^{M^{\Pp},A},M^{\Pp},S)$. It is well established that the BCs for $X^{M^{\Pp},A}$ are implemented using Neumann free BCs \eqref{eq: lower barrier PDE} and \eqref{eq: upper barrier PDE}, while the BCs for $M^{\Pp}$ impose Neumann fixed BCs because the triggers are predetermined. By combining this with a deterministic solution of $S$ to a Ricati equation, a verification theorem can be easily constructed, as shown, for example, by \citet{kushner2001numerical}. Once we have the solution, we can use it to intervene in the dynamics of $(X^{M^{\Pp},A},M^{\Pp})$ in accordance with the optimal control policy.

However, even with a solution to the HJB equation \eqref{eq: HJB3} (whether obtained analytically or numerically), it cannot be directly applied in practice. This is because we can only observe $X^{M^{\Pp}}$ and not $M^{\Pp}$. Therefore, we use a technique called \emph{coordinate transformation} from \cite{johnson2017quickest} to reduce a diffusion dimension from $M^{\Pp}$ so that only $X^{M^{\Pp}}$ exhibits stochastic evolution. The main idea is to apply a suitable transformation that converts the FBSDE system \eqref{eq: X^A_t}-\eqref{eq: J^A_t} into a different coordinate system where the (transformed) unobservable process is independent of the innovation process. We demonstrate that after this transformation, it follows a deterministic evolution dependent on both bounded variation processes $A$ and $(L^m,U^m)$. This allows us to impose the BCs in the new system in the classical sense of singular control. To do this, we define a transformation $\mathcal{T}:\sr^2\times\sr_+\to\sr^2\times\sr_+$ by 
\begin{align}\label{def: transform}
    \mathcal{T}(x,m,s)=(x,x-\frac{\sigma}{s}m,s)
\end{align}
where the invertible is given by 
\begin{align}\label{def: inverse transform}
    \mathcal{T}^{-1}(x,m,s)=(x,\frac{s}{\sigma}(x-m),s)
\end{align}
Taking $X_2\triangleq X_2(x,m,s)=x-\frac{\sigma}{s}m$, then it follows by It\={o}'s lemma that the new system of the forward SDEs $( X_1,X_2,S)\triangleq(X_{1,t},X_2(X_{1,t},M^{\Pp}_t,S_t),S_t)$ solve
\begin{align}
     &\rd  X_{1,t} &&= \bigg(\alpha( X_{1,t}) + S_t( X_{2,t} - X_{1,t} )\bigg)\rd t + \rd A_t+\sigma( X_{1,t})\rd \overline B_t,&& X_{1,0}=x, \label{eq: X1}\\
    &\rd X_{2,t} &&= \bigg(\alpha( X_{1,t})+ 2S_t( X_{2,t} - X_{1,t} )\bigg) \rd t  + \rd A_t -\frac{\sigma(X_{1,t})}{S_t}\rd D_t, && X_{2,0}=x-\frac{\sigma}{s}m, \label{eq: X2}\\
    &\rd S_t &&= -S^2_t \rd t, && S_0=s. \label{eq: S_t 2}
    \end{align}


The inverse transformed version of the value function \eqref{eq: DPP v} is defined by the mapping $\vt:[0,T]\times\sr^2\times\sr_+\to[0,T]\times\sr^2\times\sr_+$ where
\begin{align}\label{def: tran val fn}
    \vt (t,\hat x) \triangleq  v (t,\mathcal{T}^{-1}(\hat x)) = v (t,x,m,s),
\end{align}
given $\hat x=(x_1,x_2,s).$ By this transformation, $D$ now only increases when $\frac{S_t}{\sigma(X_{1,t})}(X_{1,t} - X_{2,t})= \mL$ or $\mU$, respectively. 

We observe that $X_1 = X^{ M^{\Pp},A}$, which remains consistent with the pre-transformed process. In contrast, the process $X_2$ offers an entirely new interpretation. While $X_2$ resembles $X_1$, it notably lacks exposure to risk. This transformation yields a practical advantage, as it enables us to elicit $M^{\Pp}$ from the environment by observing $X_2$, which is itself dynamically inferred from $X_1$. For this reason, we refer to $X_2$ as an \textit{auxiliary controlled inventory process}. Furthermore, from~\eqref{eq: X1} and~\eqref{eq: X2}, we obtain 
\begin{align}\label{eq: X1-X2} 
    \rd (X_{1,t}-X_{2,t}) = S_t(X_{1,t}-X_{2,t})\rd t +\sigma( X_{1,t})\rd \overline B_t+\frac{\sigma(X_{1,u})}{S_t}\rd D_t, 
\end{align} 
which implies that, in expectation over the long run, the increment of $X_1$ converges to that of $X_2$. This follows from the fact that $S$ is a strictly decreasing function of time and bounded within $(0,s]$, while $X_1$ and $X_2$ are themselves bounded by the processes $A$ and $D$, respectively. Therefore, the DM may eventually expect to recover full information about the controlled inventory process $X_1$ through sufficiently long observation.

Since by construction, $v(t,r)=\vt(t,\hat x)$, it implies that a transformed HJB equation (to be shown later) solved by $\vt(t,\mathcal{T}(\hat x))$, gives a new optimal control policy that is equivalent to that of $v(t,r)$. To see this, one can check by the inverse transformation that
\begin{align}
    \diff{ v (t,x_1,\frac{s}{\sigma}(x_1-x_2),s)}{x} &=\diff{\vt (t,\hat x)}{x_1} + \diff{\vt (t,\hat x)}{x_2} ,\\
    \diff{ v (t,x_1,\frac{s}{\sigma}(x_1-x_2),s)}{m} &=-\frac{\sigma(x_1)}{s} \diff{\vt (t,\hat x)}{x_2} ,\\
    \frac{\partial^2v(t,x_1,\frac{s}{\sigma}(x_1-x_2),s)}{\partial x^2} &=\frac{\partial^2\vt(t,\hat x)}{\partial x_1^2} + \frac{\partial^2\vt(t,\hat x)}{\partial x_2^2} + 2\frac{\partial^2\vt(t,\hat x)}{\partial x_1 \partial x_2},\\
    \cL v(t,x_1,\frac{s}{\sigma}(x_1-x_2),s) &= \diff{\vt(t,\hat x)}{t} + \bigg(\alpha(x_1) + s(x_2-x_1)\bigg) \diff{\vt(t,\hat x)}{x_1}     \nonumber\\&\hspace{0.5cm} + \bigg(\alpha(x_1) + 2s(x_2-x_1)\bigg)\diff{\vt(t,\hat x)}{x_2}-s^2\diff{\vt(t,\hat x)}{s}+\frac{\sigma^2(x)}{2}\frac{\partial^2\vt(t,\hat x)}{\partial x_1^2}\nonumber\\&\triangleq \cL \vt(t,\hat x).
\end{align}
By plugging these terms into \eqref{eq: HJB3}, we then obtained so-called a transformed HJB equation:
\begin{align}\label{eq: HJB4}
    0&=\min\Bigg\{\cL \vt(t,\hat x) + F\bigg(t,\hat x,\vt(t,\hat x),\sigma(x_1) \diff{\vt(t,\hat x)}{x_1}  \bigg),\nonumber\\&\hspace{3cm}\;\; \ell + \diff{\vt (t,\hat x)}{x_1} + \diff{\vt (t,\hat x)}{x_2},\;\; u - \diff{\vt (t,\hat x)}{x_1} - \diff{\vt (t,\hat x)}{x_2}  \Bigg\}\\
    0&= \diff{\vt (t,x_1,x_1 - \frac{\sigma(x_1)}{s} \mL )}{x_2} = \diff{\vt (t,x_1,x_1 - \frac{\sigma(x_1)}{s} \mU )}{x_2}, 
    \\ \vt(T,\hat x)&=\xi(x_1),
\end{align}
claimed earlier. In consequence, the transformed optimal control policy represented by form of inaction and action regions can be defined by $\CC_{\hat v},\LB_{\hat v},\UB_{\hat v}\subset[0,T]\times\sr$ such that
\begin{align}
    \CC_{\hat v}&\triangleq\left\{(t,x_1,x_2)\in [0,T]\times\sr: -\ell<\diff{\vt (t,\hat x)}{x_1} + \diff{\vt (t,\hat x)}{x_2}<u\right\},\label{def: continuation region 2}\\
    \LB_{\hat v}&\triangleq\left\{(t,x_1,x_2)\in [0,T]\times\sr: -\ell\geq\diff{\vt (t,\hat x)}{x_1} + \diff{\vt (t,\hat x)}{x_2}\right\}\;\;\;\text{and,}\label{def: lower stopped region 2}\\
    \UB_{\hat v}&\triangleq\left\{(t,x_1,x_2)\in [0,T]\times\sr: \qquad\;\;\diff{\vt (t,\hat x)}{x_1} + \diff{\vt (t,\hat x)}{x_2}\geq u\right\},\label{def: upper stopped region 2}
\end{align}
which basically become a two dimensional singular controls problem. As a result, a solution the HJB equation \eqref{eq: HJB4} is equivalent to that of solving
\begin{align}
    \cL \vt(t,\hat x) + F\bigg(t,\hat x,v(t,\hat x),\sigma(x_1) \diff{\vt(t,\hat x)}{x_1}\bigg)&=0\;\;\;(t,x_1,x_2)\in\CC_{\hat v} \label{eq: continue PDE 2}\\
    \ell +\diff{\vt (t,\hat x)}{x_1} + \diff{\vt (t,\hat x)}{x_2}&=0\;\;\;(t,x_1,x_2)\in\LB_{\hat v}\label{eq: lower barrier PDE 2}\\
    u - \diff{\vt (t,\hat x)}{x_1} - \diff{\vt (t,\hat x)}{x_2}&=0\;\;\;(t,x_1,x_2)\in\UB_{\hat v}.\label{eq: upper barrier PDE 2}
\end{align}
Now we modify Assumption~\ref{ast:refined not incressing L U} to define an optimal control policy for the transformed value function \eqref{def: tran val fn} and show that it is a viscosity solution to \eqref{eq: HJB4}.

\begin{definition}
    For given $\hat v\in C[0,T]\times \sr^3$ and $h_1,h_2>0$,
    \begin{align}
    \LB_{\hat v}^{h_1,h_2}&\triangleq\{(t,x_1,x_2)\in [0,T]\times\sr^2: -\ell h_1h_2\geq \phantom{-}(h_1+h_2)\hat v(t,x,m,s)-h_2\hat v(t,x_1-h_1,x_2,s)\nonumber\\&\hspace{9.cm}-h_1\hat v(t,x_1,x_2-h_2,s)\}\label{def: tran refined lower stopped region}\\
    \UB_{\hat v}^{h_1,h_2}&\triangleq\{(t,x_1,x_2)\in [0,T]\times\sr^2: \phantom{-}uh_1h_2\geq -(h_1+h_2)\hat v(t,x,m,s)+h_2\hat v(t,x_1+h_1,x_2,s)\nonumber\\&\hspace{9.cm}+h_1\hat v(t,x_1,x_2+h_2,s)\},\label{def: tran refined upper stopped region}\;\;\;\text{and,}\\
    \CC_{\hat v}^{h_1,h_2}&\triangleq[0,T]\times\sr^2\setminus (\LB_{\hat v}^{h_1,h_2}\cup\UB_{\hat v}^{h_1,h_2}).\label{def: tran refined continuation region}
\end{align}
\end{definition}

\begin{assumption}\label{ast: tran refined not incressing L U}
    Given  $v\in C[0,T]\times \sr^3$ and. For any feasible policy $A\in \cD$, there exists $h_1,h_2>0$ such that
    \begin{align}
        \int_0^T\mathds{1}_{\{(t,X^{M^{\Pp},A}_t)\in \CC_{\hat v}^{h_1,h_2}\}} \rd A_t = 0,\;\Pp\textup{--a.s.}
    \end{align}
\end{assumption}

\begin{theorem}\label{thm: viscosity 2}
    The value function defined by  \eqref{def: tran val fn} is a unique continuous viscosity solution to \eqref{eq: HJB4}.
\end{theorem}
\noindent\begin{proof}[Proof of Theorem \ref{thm: viscosity 2}]
Since $\mathcal{T}$ has a smooth transformation, it is straightforward to check that \eqref{def: tran val fn} is a unique continuous viscosity solution to \eqref{eq: HJB4} by utilising Assumption~\ref{ast: tran refined not incressing L U} and repeating the argument given in Theorem~\ref{thm: viscosity 1}. 
\end{proof}

\section{Numerical Implementation}\label{sec:comstat}
In this section, we utilise a numerical method so-called \textit{Markov chain approximation} to solve the HJB equations \eqref{eq: HJB4} of an inventory flow modelled by an arithmetic Brownian diffusion. To neglect the choice of terminal condition, we employ the change of time to turn the HJB equations \eqref{eq: HJB4} into initial value problem. We then perform comparative statics to explore the impact of smooth ambiguity on a singular control.

\subsection{Markov chain approximation}

Thus far, a general analytical solution for solving singular control problems remains unknown. However, there exists a semi-analytical solution, demonstrated by \citet{harrison1983instantaneous}, which provides a closed-form value function parameterised by unknown control barriers. These barriers can then be determined using root-finding methods such as bisection or Newton's method. A relevant application of this approach to our problem domain is discussed by \citet{archankul_thijssen_ferrari_hellmann_2025}, where they address singular control in an inventory context with maxmin ambiguity. Nonetheless, obtaining the semi-analytical solution is contingent upon the availability of an analytic solution for the first part of the HJB equation \eqref{eq: HJB4}, namely, the solution to $\cL\vt + F=0$ in closed form. Unfortunately, our problem encounters difficulty in this regard, as the quadratic term representing smooth ambiguity adjustment, expressed in $F$, introduces non-linearity, precluding an analytical solution for $\cL\vt + F=0$. Consequently, numerical approximation becomes necessary. One of the most promising methods, extensively documented in the literature for addressing singular control problems, is the Markov chain approximation (MCA), as demonstrated by \citet{kushner1991,kushner2001numerical}.


Before addressing the MCA, we consider the following assumption for mathematical tractability.
\begin{assumption}
    \begin{enumerate}
        \item $\alpha(x) = a \in \sr, \;\textup{and}\; \sigma(x) = b >0$
        \item $f(x) = \check c x^- + \hat c x^+$, where $x^{\pm} = \max\{\pm x,0\}$ and $\check c,\hat c>0.$
        \item $\tau = T-t,$ where $v(\tau,R_{\tau})|_{\tau=0}=0 $    
    \end{enumerate}
\end{assumption}
 The initial assumption implies that the controlled inventory process follows an arithmetic Brownian motion (ABM). The second means that the inventory is subject to a linear holding cost with a target value $x=0$. The last assumption
 referred to as a \textit{change of time}, where the DM now focuses on $\tau$, the duration at which the inventory level is monitored. Moreover, $v(\tau,R_{\tau})|_{\tau=0}=0$ means that there are no initial cost of inventory holding since there is no activity at the time of the valuation. This is equivalent to considering a zero terminal condition ($\xi_T=0$), i.e., there is no `bequest' cost at the terminal time of intervention.
 By these assumptions, we therefore have the following recursive utility.
\begin{align}\label{eq: value function tau}
    V(\tau,x,m)&\triangleq v(T-t,x,m) \nonumber\\&= \inf_{A\in\cD} \Ep^{\Pp}_{\cG_{T-\tau}}\Bigg[ \int_{T-\tau}^{T} e^{-\rho r}\Bigg[ \Bigg(f(X^A_r)\rd r +\ell \rd A^-_r + u \rd A^+_r \nonumber\\ &\;\;\;\;\;\;\;\;\;\;\;\;\;\;\;\;\;\;\;\;\;\;\frac{\gamma S(T-r)}{2}\left(b\diff{V(r,X^A_r,M_r)}{x} + S(T-r)\diff{V(r,X^A_r,M_r)}{m}\right)^2\Bigg)\rd r   \Bigg]\Bigg]
\end{align}
where $(X^A,M)$ solve
\begin{align}
    \rd X^A_{t} &= (a - b M_t) \rd t + b \rd \overline B_t + \rd A_t,\;\;\;X^A_{0}=x \label{eq: X abm}\\
    \rd M_t &= S(t)\rd \overline B_t + \rd D_t,\;\;\; M_0=m. \label{eq: M abm}\\
    S(t) &= \frac{s}{1+st}.
\end{align}
Therefore, by the change of times gives that $V$ defined by \eqref{eq: value function tau} is a viscosity solution to the HJB equation:
\begin{align}\label{eq: HJB tau}
    \min\bigg\{\cL_{\tau} V(\tau,r) + \frac{\gamma S(T-\tau)}{2}&\left(b\diff{V}{x}(\tau,r)+s(T-\tau)\diff{V}{m}(\tau,r)\right)^2,&&\nonumber\\&\;\;\;\;\;\;\;\;\;\;\;\;\;\;\;\;\; \ell+\diff{V}{x}(\tau,r),u-\diff{V}{x}(\tau,r)\bigg\}=0\nonumber&&\\&\;\;\;\;\;\;\;\;\;\;\;\;\;\;\;\;\;\;\;\;\;\;\;\;\;\;\;\;\;\;\;\;\;\;\;\;\;\;\;\;\;\;\;\;\;\;\;\;\;\;\;\; V(0,r)=0,
\end{align}
where 
\begin{align*}
    \cL_{\tau} V\triangleq -2\diff{V}{\tau}-\rho V + &(a-mb)\diff{V}{x}  +\frac{b^2}{2}\secdiff{V}{x} +\frac{S(T-\tau)^2}{2}\secdiff{V}{m} + S(T-\tau)b\crossdiff{V}{x}{m}.
\end{align*}

The main idea of MCA is to approximate $(X^A,M)$ by controlled discrete-time Markov chains on a finite time-state space, which is consistent with the local property of $(X^A,M)$. Namely, given $\delta,h_1,h_2>0$, and denoted by $\left(X^{\delta,h_1},M^{\delta,h_2}\right)\triangleq\left(X^{\delta,h_1}_{k},M^{\delta,h_2}_{k}\right)_{k\geq 0}$, controlled discrete-time Markov chains on discrete time-state space \begin{align}
    \mathcal S^{\delta,h_1,h_2} \triangleq \Delta([0,T])\times\Delta([\xL, \xU])\times \Delta([\mL,\mU]),
\end{align}
where
\begin{align}
    \Delta([0,T])&\triangleq \left\{T,(N_T-1)\delta,\ldots,2 \delta,\delta,0\right\}\\
    \Delta([\xL, \xU])&\triangleq \left\{\xL, \xL+h_1, \xL+ 2h_1, \ldots, \xL +(N_x-1)h_1, \xU\right\}\\
    \Delta([\mL, \mU])&\triangleq \left\{\mL, \mL+h_2, \mL+ 2h_2, \ldots, \mL +(N_m-1)h_2, \mU\right\}
\end{align}
for some $N_T,N_x,N_m\in \N$. 
The transition probabilities are denoted by 
\begin{align}
    p^{\delta,h_1,h_2}_{t,x,m}(s,y,n)\triangleq \Pp\left((X^{\delta,h_1}_{s},M^{\delta,h_2}_{s})=(y,n)|\cG^{\delta,h_1,h_2}_{t}\right),
\end{align}
where $\cG^{\delta,h_1,h_2}_{t}$ is a discrete-time filtration generated by $\Big\{X^{\delta,h_1}_{r}, M^{\delta,h_2}_{r}| r\leq t, (X^{\delta,h_1}_{t},M^{\delta,h_2}_{t})=(x,m)\Big\}$. Here, $|\xL|$ and $|\xU|$ are chosen large enough so that $-\ell\leq\frac{ V(t,x+h_1,m)- V(t,x-h_1,m)}{2h_1}\leq u$ for any $x\in\Delta([\xL, \xU])$, while $ p^{\delta,h_1,h_2}_{t,x,m}(s,y,n)$ is defined consistently with the local dynamic of \eqref{eq: value function tau}. 

We recall the value function \eqref{eq: value function tau} weakly solves the HJB equation \eqref{eq: HJB tau}. Following \citet{kushner2001numerical}, once we apply the \textit{upwind} implicit finite difference scheme\footnote{The reason we use the implicit scheme is the fact that the control barriers in our setting is time-invariant, so that we must treat time as another (virtual) state variable to have the forward equation \eqref{eq: Tv} locally consistent with the corresponding HJB equation.} to \eqref{eq: HJB tau} on the mesh $ \mathcal S^{\delta,h_1,h_2}$,  an approximation of $V$ is given by 

\begin{align}\label{eq: singular mca}
    \overline V(t,x,m)  = 
    \begin{cases}
        \mathrm T  \overline V(t,x,m) & \textup{if } m \notin \{\mL,\mU\}\\
        \overline V(t,x,m+h_2) & \textup{if } m = \mL \\
        -\overline V(t,x,m-h_2) & \textup{if } m = \mU,
    \end{cases}
\end{align}
where

\begin{align}\label{eq: Tv}
   \mathrm T  \overline V(t,x,m) &\triangleq \min\Bigg\{ e^{-\rho \Delta t^{\delta,h_1,h_2}} \sum_{s,y,n} \overline p^{\delta,h_1,h_2}_{t,x,m}(s,y,n)\overline V(s,y,n)  
     + f(x)\Delta t^{\delta,h_1,h_2} \nonumber\\&\hspace{2cm}+ \frac{\gamma S(T-t)}{2}\bigg(\partial^{h_1,h_2}_+\overline V(t,x,m) + \partial^{h_1,h_2}_-\overline V(t,x,m)\bigg)^2 \Delta t^{\delta,h_1,h_2}, \nonumber\\
    & \hspace{2cm}\sum_{s,y,n} \check p^{\delta,h_1,h_2}_{t,x,m}(s,y,n)\overline V(s,y,n) +h\ell, \;\; hu - \sum_{s,y,n} \hat p^{\delta,h_1,h_2}_{t,x,m}(s,y,n)\overline V(s,y,n)
    \Bigg\}\nonumber\\
    &\triangleq\min\{\overline V_c(t,y,n),\overline V_{\ell}(t,y,n),\overline V_{u}(t,y,n)\},
\end{align}
where
\begin{align*}
    \partial^{h_1,h_2}_+\overline V(t,x,m) &\triangleq \left( b\frac{\overline V(t,x+h_1,m) -\overline V(t,x,m)}{2h_1} +S(T-t)\frac{\overline V(t,x,m+h_2) -\overline V(t,x,m)}{2h_2}\right)^+\\
    \partial^{h_1,h_2}_-\overline V(t,x,m) &\triangleq \left(b\frac{\overline V(t,x-h_1,m) -\overline V(t,x,m)}{2h_1} +S(T-t)\frac{\overline V(t,x,m-h_2) -\overline V(t,x,m)}{2h_2}\right)^-.
\end{align*}
\begin{remark}
    Note that the nonlinear terms of \eqref{eq: Tv} are split into two components: the positive part, handled by a forward scheme, and the negative part, handled by a backward scheme. In other words, we apply a forward scheme when the approximated value function is increasing and a backward scheme when it is decreasing. This approach not only ensures the existence of a unique fixed point for \eqref{eq: Tv}, but also implies a transition probability for the original nonlinear expectation, specifically, when the transition probability is approximated by the measure $\Qp^\ast$ determined by \eqref{def: overline Q}.
\end{remark}
The first term on the RHS of \eqref{eq: singular mca} represent the next time step value function on a continuation region where its transition probability is $\overline p^{\delta,h_1,h_2}_{t,x,m}(s,y,n) = \frac{\tilde p^{\delta,h_1,h_2}_{t,x,m}(s,y,n)}{ \underaccent{\tilde}{p}^{\delta,h_1,h_2}_{t,x,m}}$ given that
\begin{align*}
    \tilde p^{\delta,h_1,h_2}_{i\delta,x,m}(s,y,n)=
    \begin{cases}
         \frac{b^2}{2h_1^2} -\frac{b S(T-s)}{h_1h_2} + \frac{(a - mb)^{\pm}}{h_1} & \textup{if }(s,y,n)=(t,x\pm h_1,m)\\
         \frac{S(T-s)^2}{2h_2^2} -\frac{b S(T-s)}{h_1h_2} & \textup{if }(s,y,n)=(t,x,m\pm h_2)\\
         \frac{b S(T-s)}{2h_1h_2} & \textup{if }(s,y,n)=(t,x\pm h_1,m\pm h_2)\\
         \frac{2}{\delta} & \textup{if }(s,y,n)=(t-\delta,x,m)\\
         0 &\textup{otherwise},
    \end{cases}
\end{align*}
\begin{align*}
    \underaccent{\tilde}{p}^{\delta,h_1,h_2}_{t,x,m} = \frac{2}{\delta} + \frac{b^2}{h_1^2}+\frac{S(T-t)^2}{h_2^2}   -\frac{bS(T-t)}{h_1h_2} + \frac{|a - mb|}{h_1} , \textup{   and}
\end{align*}
\begin{align*}
    \Delta t^{\delta,h_1,h_2} = \frac{1}{\underaccent{\tilde}{p}^{\delta,h_1,h_2}_{t,x,m}}
\end{align*}
when $ \mathrm T \overline  V(t+\delta,x,m) =\overline V_c(t,y,n)$. To ensure that $\overline p^{\delta,h_1,h_2}_{t,x,m}(s,y,n)\in[0,1]$ on $\mathcal S^{\delta,h_1,h_2}$, we assume that $\frac{h_1}{h_2}\in\left(\frac{s}{b(1+sT)},\frac{s}{b}\right)$.  On the other hand, if the minimum of \eqref{eq: singular mca} is obtained by the second or third term, the optimal policy is then to exert the lower of the upper control, respectively. That is
\begin{align}
    \check p^{\delta,h_1,h_2}_{t,x,m}(s,y,n)=
    \begin{cases}
         1 & \textup{if }(s,y,n)=(t,x+h,m)\\
         0 &\textup{otherwise}.
    \end{cases}
\end{align}
when $\mathrm T  \overline V(t+\delta,x,m) =\overline V_{\ell}(t,y,n)$, or 
\begin{align}
    \hat p^{\delta,h_1,h_2}_{t,x,m}(s,y,n)=
    \begin{cases}
         1 & \textup{if }(s,y,n)=(t,x-h,m)\\
         0 &\textup{otherwise}.
    \end{cases}
\end{align}
when $\mathrm T \overline V(t,x,m) =\overline V_u(t,y,n)$. In other words, when the chain $X^{\delta,h}$ reaches the lower control trigger, it is reflected upward almost surely. The similar feature applies to the upper region. Therefore, the transition probability of the filtered controlled processes is defined by 
\begin{align}
    p^{\delta,h_1,h_2}_{t,x,m}(s,y,n) &\triangleq \Bigg(\overline p^{\delta,h_1,h_2}_{t,x,m}(s,y,n) \mathds{1}_{\Big\{-\ell<\frac{\overline V(t,x+h_1,m)- \overline V(t,x-h_1,m)}{2h_1}<u\Big\}} \nonumber\\&\;\;\;\;\;\;\;\;\;\;+ \check p^{\delta,h_1,h_2}_{t,x,m}(s,y,n) \mathds{1}_{\Big\{\frac{\overline V(t,x+h_1,m)- \overline V(t,x,m)}{h_1}\leq \ell\Big\}} \nonumber\\&\;\;\;\;\ \;\;\;\;\;+ \hat p^{\delta,h_1,h_2}_{t,x,m}(s,y,n) \mathds{1}_{\Big\{\frac{\overline V(t,x,m)- \overline V(t,x-h_1,m)}{h_1}\geq u\Big\}}\Bigg)\mathds{1}_{m\notin\{\mL,\mU\}} + \mathds{1}_{m=\mL} + \mathds{1}_{m=\mU}.
\end{align}
Define $(\Delta X^{\delta,h_1}_{t},\Delta M^{\delta,h_2}_{t}) \triangleq ( X^{\delta,h_1}_{t+\delta} - X^{\delta,h_1}_{t}, M^{\delta,h_2}_{t+\delta}-M^{\delta,h_2}_{t})$, and denote $\Ep^{\delta,h_1,h_2}_{t,x,m}$ a conditional expectation consistent with $p^{\delta,h_1,h_2}_{t,x,m}(s,y,n)$. By convexity of $v$, one can check that 
\begin{align}
    \sum_{s,y,n}p^{\delta,h_1,h_2}_{t,x,m}(s,y,n) &=1\\
    \Ep^{\delta,h_1,h_2}_{t,x,m}[\Delta X^{\delta,h}_{t}] &=(a-M^{\delta,h}_tb)\Delta t^{\delta,h_1,h_2} + \mathcal O(\delta,h) \label{eq: x mca mean}\\
    \Ep^{\delta,h_1,h_2}_{t,x,m}[(\Delta X^{\delta,h}_{t} -  \Ep^{\delta,h_1,h_2}_{x,m}[\Delta X^{\delta,h}_{t}])^2]&= b^2\Delta t^{\delta,h_1,h_2} + \mathcal O(\delta,h)\label{eq: x mca variance}\\
    \Ep^{\delta,h_1,h_2}_{t,x,m}[\Delta M^{\delta,h}_{t}] &=0 + \mathcal O(\delta,h) \label{eq: m mca mean}\\
    \Ep^{\delta,h_1,h_2}_{t,x,m}[(\Delta X^{\delta,h}_{t} -  \Ep^{\delta,h_1,h_2}_{x,m}[\Delta X^{\delta,h}_{t}])^2]&= S(T - t)^2\Delta t^{\delta,h_1,h_2} + \mathcal O(\delta,h)\label{eq: m mca variance}
\end{align}
where $ \mathcal O(\delta,h)$ is the error due to approximated $\delta$ and $h$. This in fact align with the local consistency of \eqref{eq: X abm} and \eqref{eq: M abm}, which means $\lim_{(\delta,h)\downarrow(0,0)}(X^{\delta,h},M^{\delta,h})=(X^A,M)$, $\Pp$--a.s. (cf. \citet{kushner2001numerical}). 

Next, we demonstrate that the transformation $\mathrm T$ has a unique fixed point, ensuring that $\overline V$ provides a unique approximation of $V$. The following lemmas play a crucial role in this proof.

\begin{lemma}\label{lem: ineq min}
    Suppose that $a_i,b_i\in\sr$ for $i=1,2,\dots,N$, for some $N\in\N$. Then the following inequality holds.
    \begin{align}\label{eq: ineq min}
        \bigg|\min_{i=1,2,\dots,N}a_i - \min_{i=1,2,\dots,N}b_i\bigg|\leq \min_{i=1,2,\dots,N}|a_i-b_i|.
    \end{align}
\end{lemma}

\noindent\begin{proof}[Proof of Lemma~\ref{lem: ineq min}]
    Suppose that that $i=1,2$. Then, it holds that
    \begin{align*}
        |\min\{a_1,a_2\}-\min\{b_1,b_2\}|=\begin{cases}
            |a_1-b_1| & \textup{if}\;\; \min\{a_1,a_2\}=a_1\geq\min\{b_1,b_2\}=b_1\\
            |a_1-b_2|\leq |a_2-b_2| & \textup{if}\;\; \min\{a_1,a_2\}=a_1\geq\min\{b_1,b_2\}=b_2\\
            |a_2-b_1|\leq |a_1-b_1| & \textup{if}\;\; \min\{a_1,a_2\}=a_2 \geq\min\{b_1,b_2\}=b_1\\
            |a_2-b_2| & \textup{if}\;\; \min\{a_1,a_2\}=a_2\geq\min\{b_1,b_2\}=b_2.
        \end{cases}
    \end{align*}
    In summary, we have $|\min\{a_1,a_2\}-\min\{b_1,b_2\}|\leq\min\{|a_1-b_1|,|a_2-b_2|\}$. Similarly, it follows that \eqref{eq: ineq min} holds for any $N\in \N$, therefore, completing the proof.
\end{proof}

By a straightforward calculation, we obtain the following result.
\begin{lemma}\label{lem: T hat}
    Suppose that $\widehat{\mathrm T}$ is an operator such that
    \begin{align}\label{eq: Tv2}
   \widehat{\mathrm T}  \overline V(t,x,m) &\triangleq \min\Bigg\{ e^{-\rho\Delta t_q^{\delta,h_1,h_2}(s,y,n)}\sum_{s,y,n} \overline q^{\delta,h_1,h_2}_{t,x,m}(s,y,n)\overline V(s,y,n)  
     + f(x)\Delta  t_q^{\delta,h_1,h_2}(s,y,n), \nonumber\\
    & \hspace{1cm}\sum_{s,y,n} \check p^{\delta,h_1,h_2}_{t,x,m}(s,y,n)\overline V(s,y,n) +h\ell, \;\; hu - \sum_{s,y,n} \hat p^{\delta,h_1,h_2}_{t,x,m}(s,y,n)\overline V(s,y,n)
    \Bigg\},
\end{align}
where 
\begin{align}
    \overline q^{\delta,h_1,h_2}_{t,x,m}(s,y,n) &= \frac{\tilde p^{\delta,h_1,h_2}_{t,x,m}(s,y,n)+\tilde q^{\delta,h_1,h_2}_{t,x,m}(s,y,n)}{ \underaccent{\tilde}{p}^{\delta,h_1,h_2}_{t,x,m}+\underaccent{\tilde}{q}^{\delta,h_1,h_2}_{t,x,m}(s,y,n)}\label{eq: q hat}\\
    \tilde q^{\delta,h_1,h_2}_{t,x,m}(s,y,n)&= \begin{cases}
         \partial^{h_1,h_2}_\pm\overline V(t,x,m)\frac{\gamma b S(T-t)}{2h_1} & \textup{if }(s,y,n)=(t,x\pm h_1,m)\\
         \partial^{h_1,h_2}_\pm\overline V(t,x,m)\frac{\gamma S^2(T-t)}{2h_2} & \textup{if }(s,y,n)=(t,x,m\pm h_2)\\
         0 &\textup{otherwise},
    \end{cases}\nonumber\\
    \underaccent{\tilde}{q}^{\delta,h_1,h_2}_{t,x,m}(s,y,n)&= \left(\partial^{h_1,h_2}_+\overline V(t,x,m)+\partial^{h_1,h_2}_-\overline V(t,x,m)\right)\left(\frac{b}{h_1}+\frac{S(T-t)}{h_2}\right)\frac{\gamma S(T-t)}{2}.\nonumber\\
    \Delta  t_q^{\delta,h_1,h_2}(s,y,n)&=\frac{1}{ \underaccent{\tilde}{p}^{\delta,h_1,h_2}_{t,x,m}+\underaccent{\tilde}{q}^{\delta,h_1,h_2}_{t,x,m}(s,y,n)} \nonumber
\end{align}
Then $\mathrm T = \widehat{\mathrm T}$. Moreover, $\overline q^{\delta,h_1,h_2}_{t,x,m}(s,y,n)\in[0,1)$, for all $(s,y,n)\in \mathcal S^{\delta,h_1,h_2}$.
\end{lemma}

Now, we are ready to show the uniqueness of $\overline V$.

\begin{proposition}\label{prop: unique V}
    A recursive functional $\overline V$ defined by \eqref{eq: singular mca} has a unique solution.
\end{proposition}

\noindent\begin{proof}[Proof of Proposition~\ref{prop: unique V}]
    Suppose  that $\overline V_1$ and $\overline V_2$ are solutions of  \eqref{eq: singular mca} associated with time step and transition probabilities $\big(\Delta  t_{q,i}^{\delta,h_1,h_2},\overline q^{\delta,h_1,h_2}_{i,t,x,m},\check p^{\delta,h_1,h_2}_{i,t,x,m},\hat p^{\delta,h_1,h_2}_{i,t,x,m}\big),\;i=1,2$, each satisfies \eqref{eq: q hat}. Then by Lemma~\ref{lem: ineq min} and \ref{lem: T hat}, it follows that for any $(t,x,m)\in \mathcal S^{\delta,h_1,h_2}$,
    \begin{align*}
        \Big|\mathrm T\overline V_1(t,x,m) &- \mathrm T\overline V_2(t,x,m)\Big|=\Big|\widehat{\mathrm T}\overline V_1(t,x,m) - \widehat{\mathrm T}\overline V_2(t,x,m)\Big|\\&\leq\min\Bigg\{\Bigg|e^{-\rho \Delta t^{\delta,h_1,h_2}_{q,1}(s,y,n)} \sum_{s,y,n} \overline q^{\delta,h_1,h_2}_{1,t,x,m}(s,y,n)\overline V_1(s,y,n)-\\&\;\;\;\;\;\;\;\;\;\;\;\;\;\;\;\;\;\;\;e^{-\rho \Delta t^{\delta,h_1,h_2}_{q,2}}(s,y,n) \sum_{s,y,n} \overline q^{\delta,h_1,h_2}_{2,t,x,m}(s,y,n)\overline V_2(s,y,n)\Bigg|\\ &\;\;\;\;\;\;\;\;\;\;\;\;\;\;\;\;\;\;\; \Bigg|\sum_{s,y,n} \check p^{\delta,h_1,h_2}_{1,t,x,m}\overline V_1(s,y,n)-\sum_{s,y,n} \check p^{\delta,h_1,h_2}_{2,t,x,m}\overline V_2(s,y,n)\Bigg|,\\&\;\;\;\;\;\;\;\;\;\;\;\;\;\;\;\;\;\;\;\Bigg|\sum_{s,y,n} \hat p^{\delta,h_1,h_2}_{1,t,x,m}\overline V_1(s,y,n)-\sum_{s,y,n} \hat p^{\delta,h_1,h_2}_{2,t,x,m}\overline V_2(s,y,n)\Bigg|
        \Bigg\}\\
        &\leq\min\bigg\{e^{-\rho \min_{i=1,2} \Delta  t^{\delta,h_1,h_2}_{q,2}(s,y,n) },1\bigg\}\Big| \overline V_1(s,y,n)-\overline V_2(s,y,n)\Big|\\
        &=e^{-\rho \min_{i=1,2} \Delta t_i^{\delta,h_1,h_2}}\Big| \overline V_1(s,y,n)-\overline V_2(s,y,n)\Big|,
    \end{align*}
    for all $(s,y,n)\in \mathcal S^{\delta,h_1,h_2}$. The second equality holds because $\overline q^{\delta,h_1,h_2}_{i,t,x,m}(s,y,n), \check p^{\delta,h_1,h_2}_{i,t,x,m},\check p^{\delta,h_1,h_2}_{i,t,x,m}\in[0,1]$ for $i=1,2$, and any $(s,y,n)\in \mathcal S^{\delta,h_1,h_2}$, thanks to Lemma~\ref{lem: T hat}. Clearly, $e^{-\rho \min_{i=1,2} \Delta t_i^{\delta,h_1,h_2}}\in[0,1)$ since $\Delta t_i^{\delta,h_1,h_2}\in(0,1)$ for $i=1,2$  and $\rho>0$. This means $\mathrm T$ is a contraction mapping. Hence, by the contraction mapping theorem (cf. \citet{Aliprantis2006-mo}, theorem 3.48), $\mathrm T$ has has a unique fixed point and therefore $\overline V$ uniquely determine an approximation of $V$. 
\end{proof}


\subsection{Comparative statics of Arithmetic Brownian Motion}
Now, we illustrate numerical results, obtained from the MCA in the case of an ABM inventory control, to investigate the impact of ambiguity levels on the optimal control policy with observation time $T=20$ and initial condition $\overline V(0,x,m)=0$ for any $(\tau,x,m)\in \mathcal S^{\delta,h_1,h_2}$, where $(\xL,\xU)$ is chosen large enough to contain $\CC^{h_1}_v$ in $\mathcal S^{\delta,h_1,h_2}$, and $(\mL,\mU)$ is selected to contains the largest $100(1-d)\%$ confidence interval of $M_t$ for all $t\in[0,T]$. That is, it must holds that 
\begin{align*}
    (\mL,\mU) \supseteq \sup_{t\geq0}\left(m-z_{\frac{d}{2}}\sqrt{\frac{s^2t}{1+st}}, m+z_{\frac{d}{2}}\sqrt{\frac{s^2t}{1+st}}\right) = \left(m-z_{\frac{d}{2}}s,m+z_{\frac{d}{2}}s\right)
\end{align*}
where $z_{\frac{d}{2}}$ is the standard normal quantile function.For simplicity, we assume henceforth that the decision maker operates with a $99\%$ confidence level (i.e., $z_{0.005} \approx 2.576$) over the zero mean reference prior, i.e., $m=0$. Through trial and error, we find that setting $(\xL, \xU) = (-30, 30)$ and $(\mL, \mU) = (-10, 10)$ yields a suitable choice of boundaries for the model parameters. These values are therefore fixed for the remainder of our analysis.

\subsubsection{General characteristic of the control policy}\label{sec: general}
First, we demonstrate how learning and ambiguity influence the value function and the associated optimal control barriers.

Figure~\ref{fig:vf_surface} shows surface plots of value function in different levels of ambiguity given $\rho=0.2,a=0.2,b=0.2,s=0.1,\ell=u=2$ and $\check c =\hat c =2$ at observation period $\tau=20$. The $x$-axes defines the observed inventory level $X_1$, while the $y$-axes determines belief process $M$ which is here transformed to $M\triangleq\frac{S}{b}(X_1-X_2)$, specifying the interplay $X_1$ and its deterministic counterpart $X_2$. As anticipated, increasing $\gamma$ leads to an increase in the value function. This occurs because the DM accounts for a higher expected cost by evaluating certainty equivalence over posterior estimation, rather than simply relying on the standard expected utility. In other words, more cautious observation results in increased cost. Moreover, the control barrier is unique for each $\gamma$, since the value function is strictly convex in $X_1$ for any $M$. Later, we show that the increase in value function and the shape of control barriers result from the DM's response to ambiguity. 

\begin{figure}[!htb]
   \begin{subfigure}{0.48\textwidth}
     \centering
     \includegraphics[width=\linewidth]{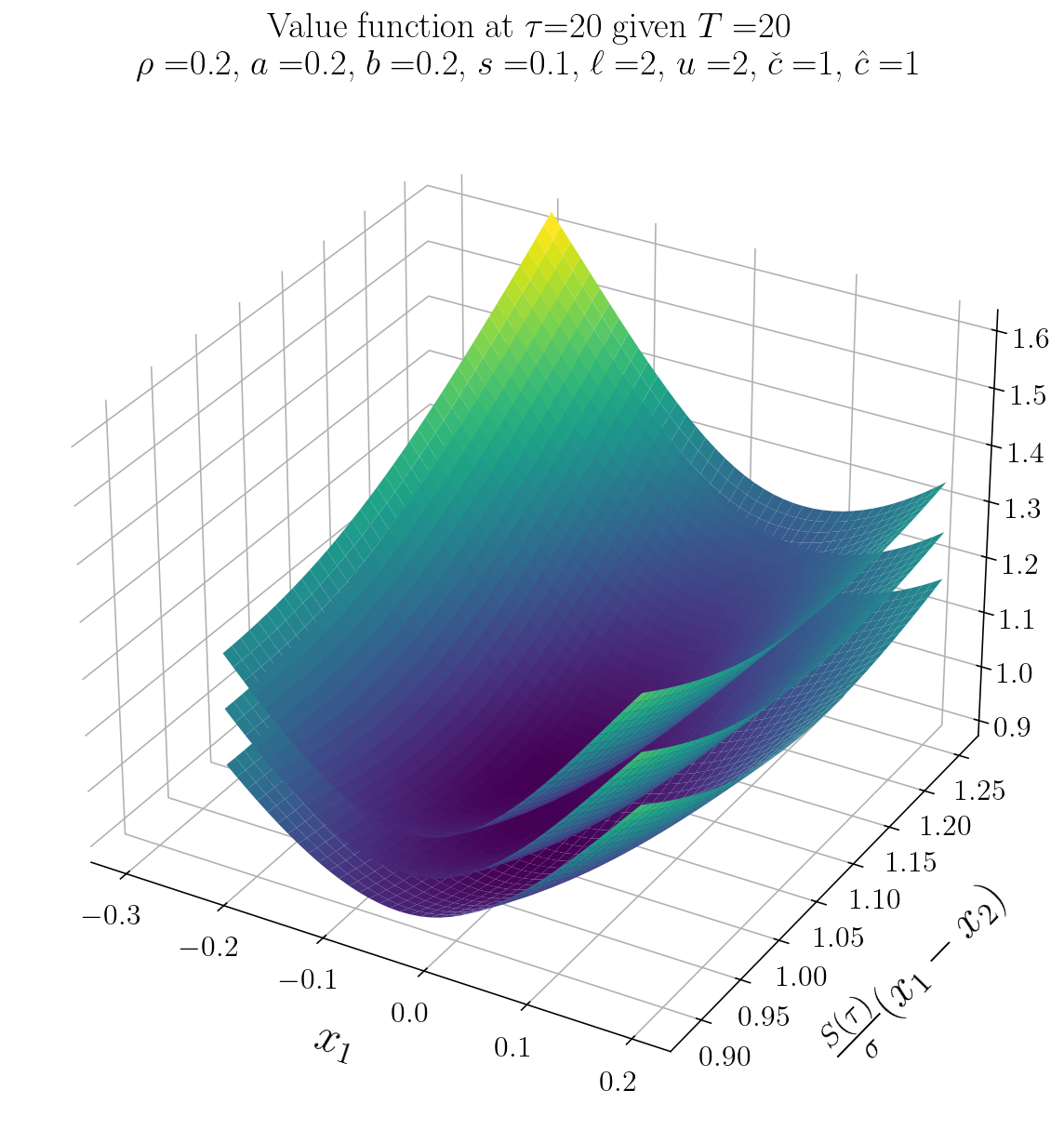}
     \caption{}\label{fig:vf_surface}
   \end{subfigure}\hfill
   \begin{subfigure}{0.48\textwidth}
     \centering
     \includegraphics[width=\linewidth]{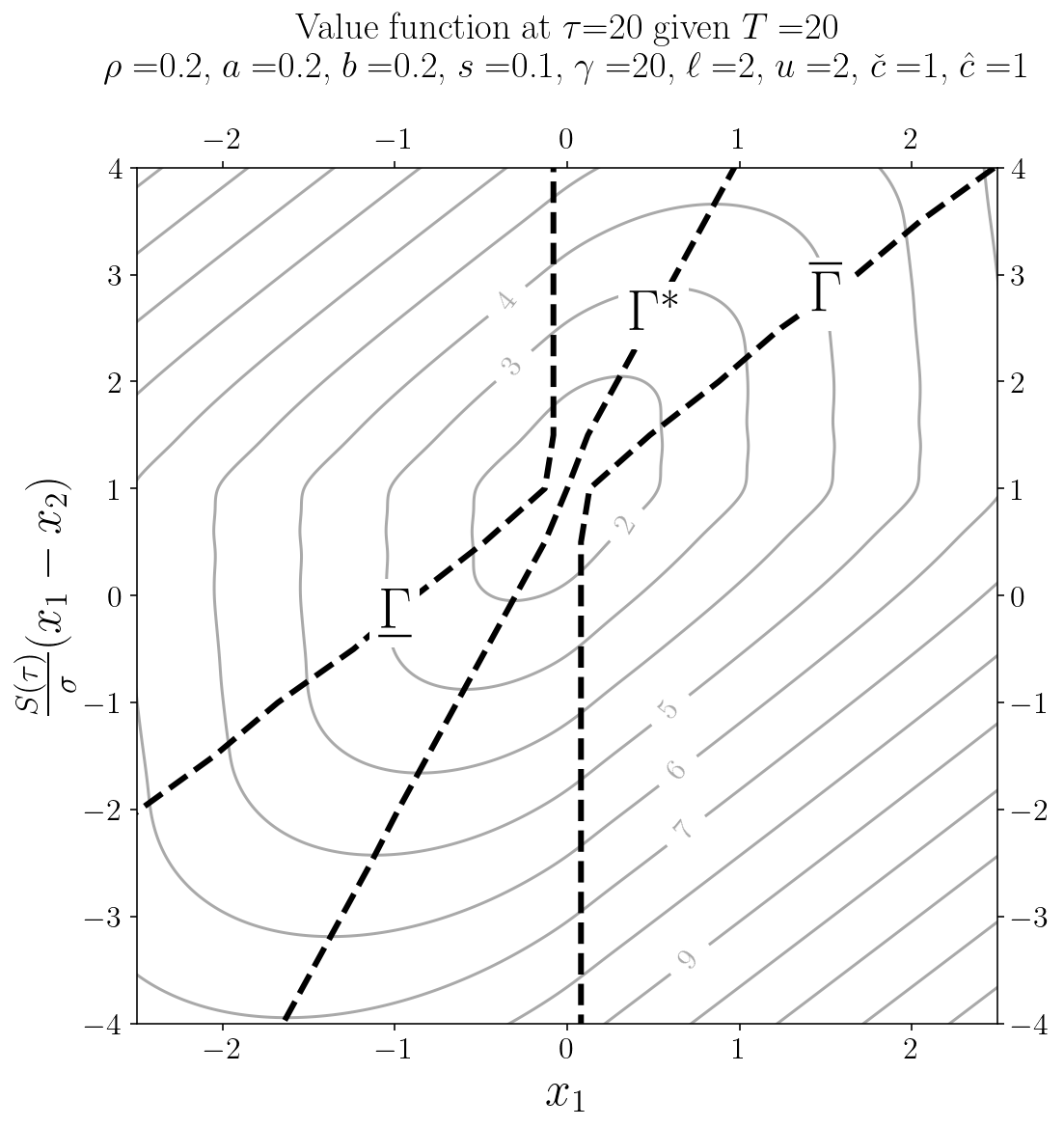}
     \caption{}\label{fig:vf+cb_contour}
   \end{subfigure}
   \caption{Value functions with fixed parameters $\rho=0.2,a=0.2,b=0.2,s=0.1,\ell=u=2, \check{c}=\hat{c}=1$ and $T=20$ at $\tau=20$. Panel (\subref{fig:vf_surface}) displays the surface plots for $\gamma\in\{0,20,40\}$ in an ascending order, while panel (\subref{fig:vf+cb_contour}) demonstrates the contour plot and control barriers for $\gamma=20$. Here $\underline \Gamma$, $\overline \Gamma$ and $\Gamma^\ast$ denote the lower barrier, the upper barrier and the minimum points along the process $X_1$, respectively.}
   \label{fig:vf}
\end{figure}


Figure~\ref{fig:vf+cb_contour} displays contour plots of the value function with $\gamma=20$ at $\tau=20$, along with the optimal control barriers. The left and right dashed lines represent the lower and upper control barriers, respectively, while the middle dashed line indicates the saddle points of the value function for each $M$. The entire area between the lower and upper control barriers defines the \textit{continuation region,} where the DM must apply singular control (in the $(0,1)$-direction) once $X_1$ reaches these barriers to maintain the content within the continuation region for optimality.

It is evident that the minimum (target) point of the value function lies within the narrowest continuation region, which is where $X_1$ has a zero trend, i.e., $a -bM=a+S(X_2-X_1)=0$. The implication is that the decision maker should apply maximal intervention as $X_1$ approaches this neighbourhood in order to minimise costs, since, on average, it incurs the least holding costs for both negative and positive inventory level.

When $\frac{S}{b}(X_1-X_2)$ deviates from the target point, there are two types of control policies to consider: one when $X_1$ and $X_2$ move in the same direction, and another when they move in opposite directions. To understand this, let us first examine the case where both $X_1$ and $X_2$ are negative but not far apart. In this scenario, the optimal control policy, suggested by Figure~\ref{fig:cb_gamma} to~\ref{fig:cb_rho}, indicates that it is best to \textit{wait and learn} how $X_1$ evolves. The same principle applies when both $X_1$ and $X_2$ are positive. This inference is supported by \eqref{eq: X1-X2}, which show that $X_1$ and $X_2$ tend to converge towards each other in a time. As long as the observed inventory level is not too far from its deterministic counterpart (in either direction), it is optimal to do nothing, since the inventory level is expected to reach its target sooner.
However, when $X_1$ and $X_2$ become significantly different, despite having the same negative sign, it is optimal to apply the lower control if $X_1$ is negative. This is because the difference between $X_1$ and $X_2$ leads to a higher value of $M$, which results in a downward trend that accelerates the decline in inventory content, thereby increasing the cost of holding a negative inventory. Therefore, the lower control must be applied. The same logic applies to the upper control when $X_1$ becomes positive.

However, when $X_1$ and $X_2$ move in opposite directions, even with a small difference, the wait and learn approach described earlier does not apply. For instance, consider the case where $X_1$ is negative but $X_2$ is positive. In this scenario, regardless of the magnitude of $X_2$, once $X_1$ decreases in content, it is necessary to immediately implement the lower control.
There are two reasons for this. Firstly, such a situation leads to a significant increase in $M$, which ultimately triggers an immediate application of the lower control, similar to the earlier observation. Secondly, this scenario contradicts the mutual reversion tendency of $X_1$ and $X_2$, as indicated by \eqref{eq: X1-X2}. Therefore, a substantial deviation between them signifies a \textit{false update}, necessitating the lower control to bring $X_1$ back to its target level.
The same reasoning applies to the upper control when $X_1$ becomes positive while $X_2$ is negative.

Given the distinct differences in learning and control policy rules between the two cases, we refer to the first scenario as the \textit{learning-dominant} region and the second as the \textit{control-dominant} region. 

Now we investigate the shape of these regions in the light of ambiguity.

\subsubsection{Comparative Statics of ambiguity attitude}

\begin{figure}[!htb]
   \begin{subfigure}{0.48\textwidth}
     \centering
     \includegraphics[trim={0 0 0 0.9cm},clip,width=\linewidth]{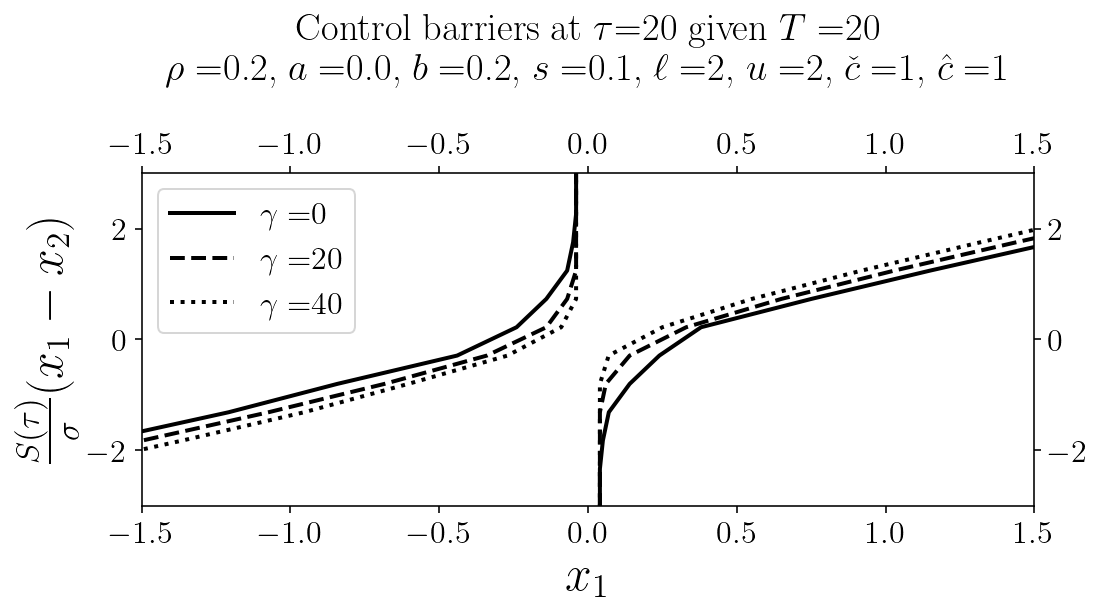}
     \caption{}\label{fig:cb_gamma}
   \end{subfigure}\hfill
   \begin{subfigure}{0.48\textwidth}
     \centering
     \includegraphics[trim={0 0 0 0.9cm},clip,width=\linewidth]{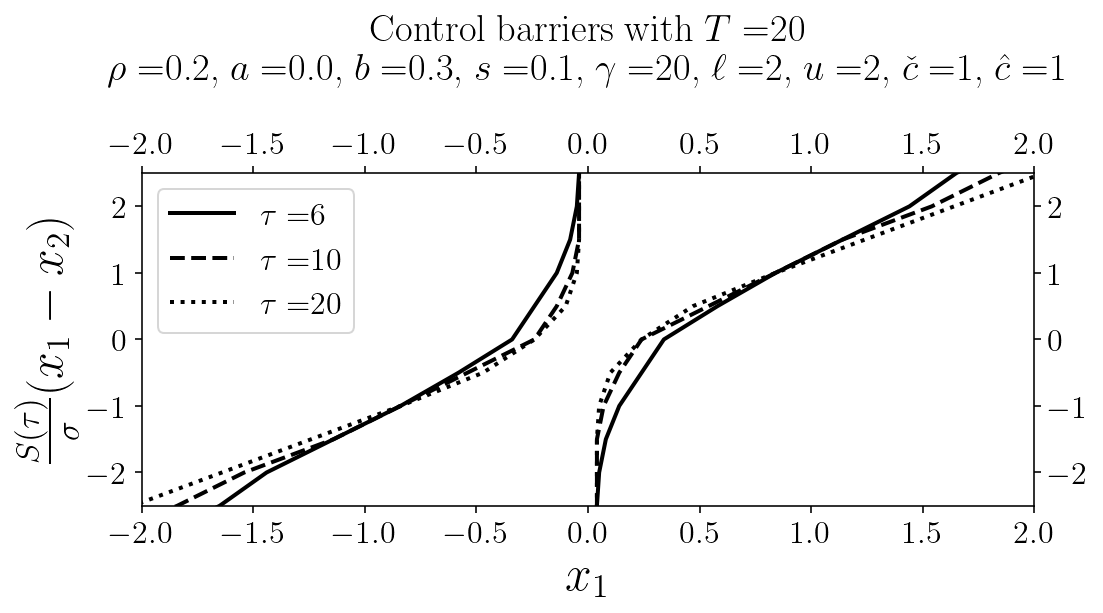}
     \caption{}\label{fig:cb_tau}
   \end{subfigure}\hfill
   \begin{subfigure}{0.48\textwidth}
     \centering
     \includegraphics[trim={0 0 0 0.9cm},clip,width=\linewidth]{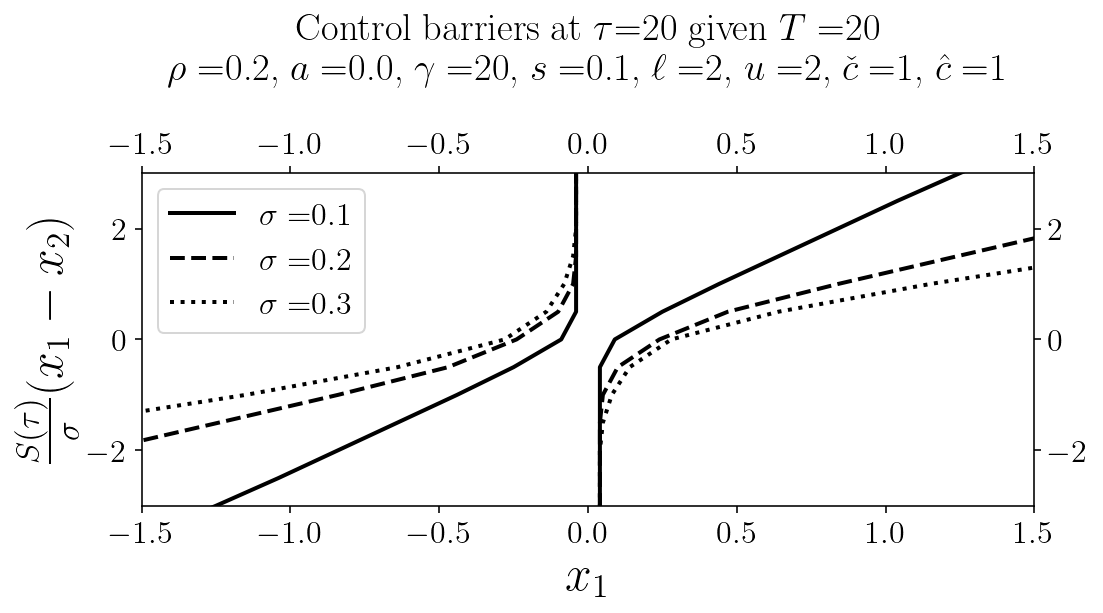}
     \caption{}\label{fig:cb_sigma}
   \end{subfigure}\hfill
   \begin{subfigure}{0.48\textwidth}
     \centering
     \includegraphics[trim={0 0 0 0.9cm},clip,width=\linewidth]{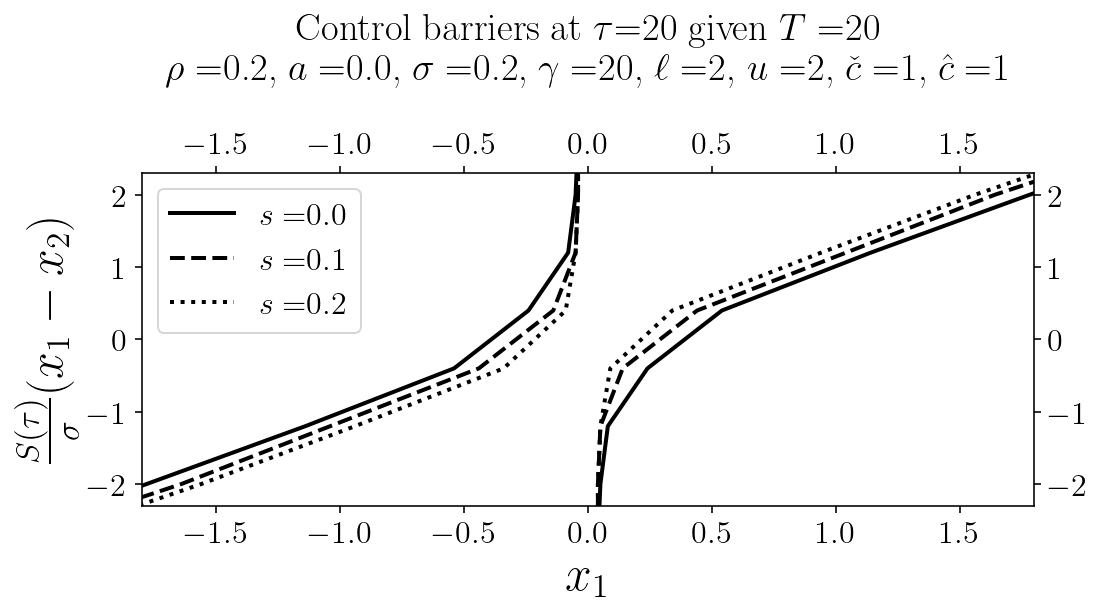}
     \caption{}\label{fig:cb_s0}
   \end{subfigure}\hfill
   \begin{subfigure}{0.48\textwidth}
     \centering
     \includegraphics[trim={0 0 0 0.9cm},clip,width=\linewidth]{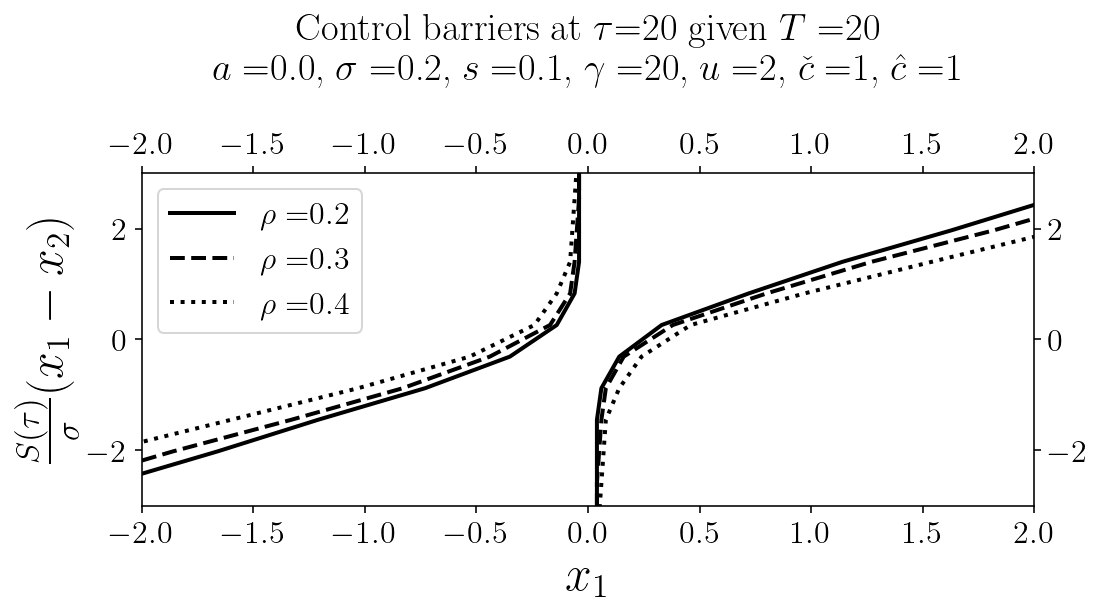}
     \caption{}\label{fig:cb_rho}
   \end{subfigure}\hfill
   \begin{subfigure}{0.48\textwidth}
     \centering
     \includegraphics[trim={0 0 0 0.9cm},clip,width=\linewidth]{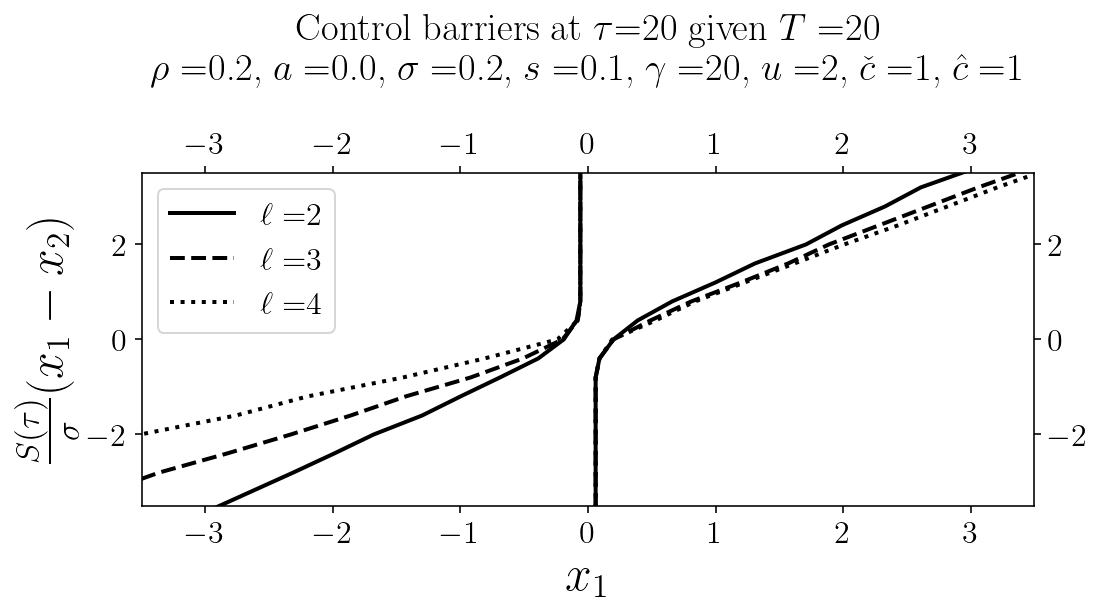}
     \caption{}\label{fig:cb_l}
   \end{subfigure}\hfill
   \begin{subfigure}{0.48\textwidth}
     \centering
     \includegraphics[trim={0 0 0 0.9cm},clip,width=\linewidth]{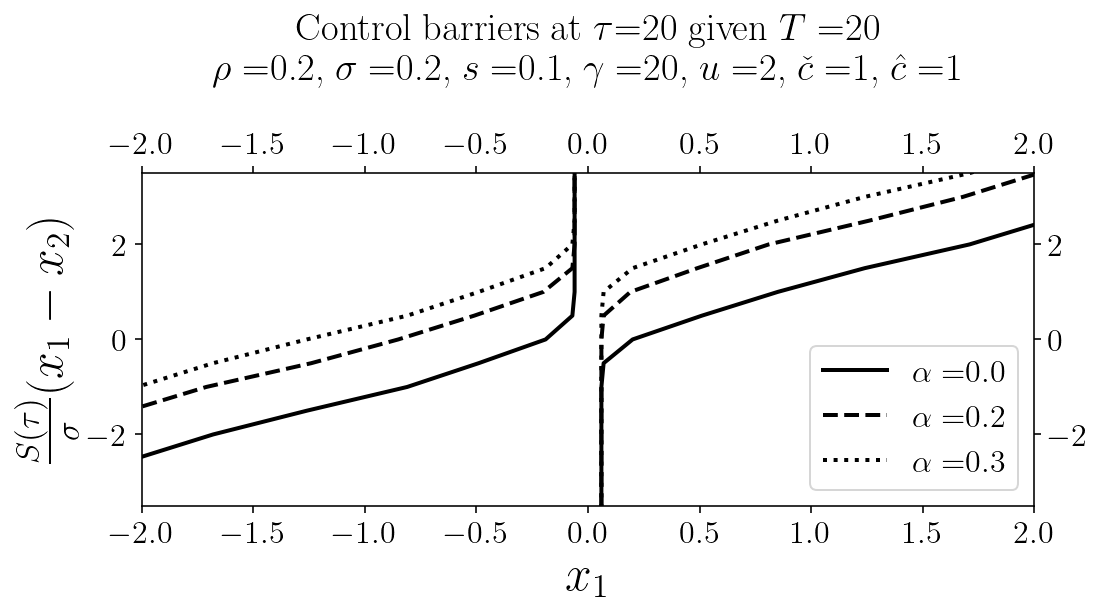}
     \caption{}\label{fig:cb_alpha}
   \end{subfigure}\hfill
   \begin{subfigure}{0.48\textwidth}
     \centering
     \includegraphics[trim={0 0 0 0.9cm},clip,width=\linewidth]{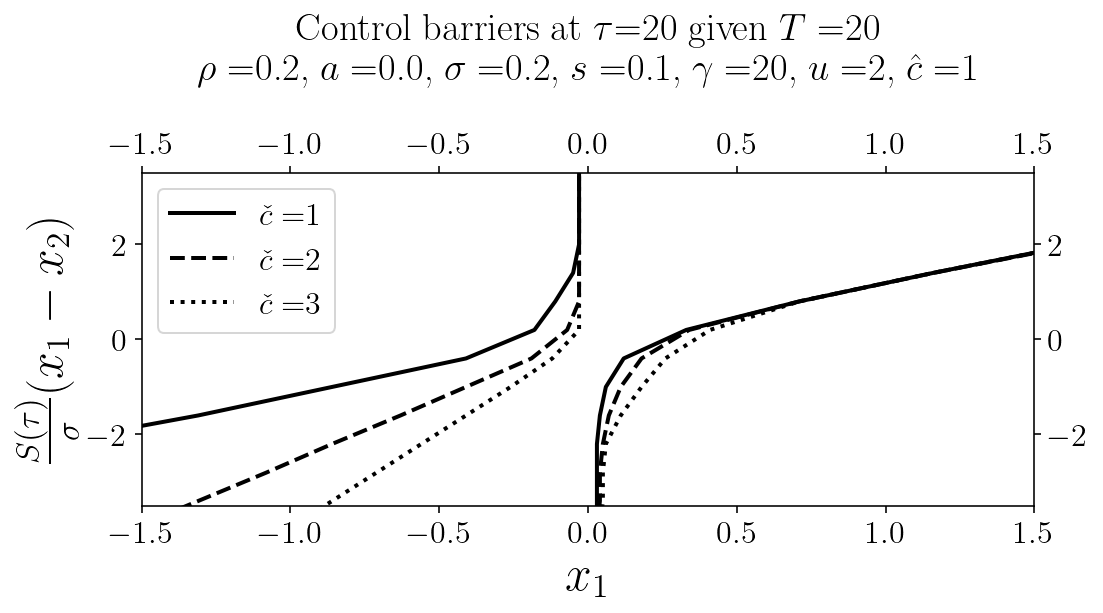}
     \caption{}\label{fig:cb_cc}
   \end{subfigure}\hfill
   \caption{Control barriers of value functions of terminal time $T=20$.}
   \label{fig:cb}
\end{figure}

The comparative statics of ambiguity attitude is demonstrated by Figures~\ref{fig:cb_gamma}, where $\rho=0.2,a=0.2,b=0.2,s=0.1,\ell=u=2, \check{c}=\hat{c}=1$ and $\tau=20$ are fixed, while $\gamma\in\{0,20,40\}$ is considered.
Evidently, ambiguity aversion reduces the size of the continuation regions at any given time to maturity. In other words, a more ambiguous DM tends to take action sooner, on average. The fact that this earlier action increases the value function suggests that the ambiguous DM is expected to act more frequently than an overconfidence one. This, in turn, implies that ambiguity amplifies trends: $X_1$ rises (or falls) more quickly when there is an upward (or downward) change in the inventory level, indicating that ambiguity incorporates a worst-case posterior update to estimate the model $\theta$ from the reference measure\footnote{This trends also align with the findings from \citet{archankul_thijssen_ferrari_hellmann_2025} in the case of maxmin utility, verifying the decision making of singular control under ambiguity}. 

Next, we investigate how the observation time shapes the control poliy. 


\subsubsection{Comparative statics of observation time}
Figure~\ref{fig:cb_tau} illustrates the shapes of the control barriers for observation times $\tau \in \{6,10,20\}$, with parameters $\rho = 0.2$, $a = 0.2$, $b = 0.2$, $s = 0.1$, $\gamma=20$, $\ell = u = 2$, and $\check{c} = \hat{c} = 1$. As observation time increases, the target region shrinks, while the learning-dominant region expands. This indicates that the DM becomes increasingly confident about the ambiguity, and thus about the target region. Accordingly, once the inventory level enters the target region, the DM should act with maximum effort to minimise the running cost. This is justified by the convergence of $X_1$ to $X_2$ over longer observation times, as suggested by~\eqref{eq: X1-X2}. Conversely, when the inventory lies within the learning-dominant region, the DM should defer intervention, as the likelihood of $X_1$ deviating from $X_2$ diminishes with time. In other words, since $X_1$ is expected to converge to $X_2$ over time, exercising control early would incur unnecessary costs.

Although there is currently no empirical evidence that directly characterises DM behaviour within the learning- and control-dominant regions, our framework provides a basis for determining the optimal control strategy, assuming the DM follows a singular control model under smooth ambiguity.

In what follows, we study the comparative static of the remaining parameters with fixed $\gamma=20$ and $\tau=20$, starting from the different levels of risk. 

\subsubsection{Comparative statics of risk}
Now, we consider the value function of $b\in\{0.1,0.2,0.3\}$ with $\rho = 0.2$, $a = 0.2$, $s = 0.1$, $\ell = u = 2$, and $\check{c} = \hat{c} = 1$. The control barriers corresponding to these parameters are illustrated in Figure~\ref{fig:cb_sigma}. We observe that, aside from the control-dominant region, the barriers expand as the level of risk increases. This is because greater risk amplifies the probability of extreme events, thereby increasing the likelihood of more frequent control interventions. As a result, it becomes more advantageous to delay control actions.

\subsubsection{Comparative statics of belief variance}
We now examine the value function for $s \in \{0, 0.1, 0.2\}$, using parameters $\rho = 0.2$, $a = 0.2$, $b = 0.2$, $\ell = u = 2$, and $\check{c} = \hat{c} = 1$. The corresponding control barriers are depicted in Figure~\ref{fig:cb_s0}. We note that when $s = 0$, the process $M$ is constant, implying the absence of ambiguity and learning. In this case, the control barrier reflects the classical singular control problem for the inventory process~\eqref{eq: XA}, with drift $\alpha(x) = a - bm$ and diffusion $\sigma(x) = b$, where $m \in (\mL, \mU)$. As the belief variance $s$ increases, we observe that the control barriers shrink, similarly to the effects of ambiguity. This occurs because a higher belief variance increases the likelihood that the belief mean deviates from the initial reference point, $M_0 = 0$. Such deviations lead to greater variability in the drift, raising the probability of incurring higher holding costs on both ends. Consequently, the optimal policy is to intervene earlier to mitigate these potential costs.

\subsubsection{Comparative statics of discounted rate}
The comparative statics of the discount rate are illustrated in Figure~\ref{fig:cb_rho}, where parameters are fixed at $a = 0.2$, $b = 0.2$, $s = 0.1$, $\ell = u = 2$, and $\check{c} = \hat{c} = 1$,  while the discount rate varies with $\rho \in \{0.2, 0.3, 0.4\}$. As the figure shows, the control barriers expand as the discount rate increases. This occurs because a higher discount rate reduces the present value of future inventory holding costs, making it relatively less costly to delay intervention. As a result, the DM has a greater incentive to wait, knowing that postponing control is, in relative terms, more cost-effective than acting immediately under a lower discount rate.

\subsubsection{Comparative statics of control cost}
We now analyse the shape of the control barriers under varying lower control costs, with $\ell \in \{2,3,4\}$ and parameters $\rho = 0.2$, $a = 0.2$, $b = 0.2$, $s = 0.1$, $u = 2$, and $\check{c} = \hat{c} = 1$. The case of higher upper control costs can be interpreted analogously. Intuitively, an increase in the lower control cost raises the expected running cost. As a result, it becomes optimal to delay intervention at the lower barrier. Moreover, since exercising control on the lower side is now relatively more costly than on the upper side, the DM should also defer action on the upper barrier to reduce the likelihood of the inventory drifting into the lower region. This interpretation is confirmed by the patterns observed in Figure~\ref{fig:cb_l}.

\subsubsection{Comparative statics of drift}
We now investigate the value function under varying drift coefficients, with $a \in \{0, 0.2, 0.3\}$, while fixing the parameters at $\rho = 0.2$, $b = 0.2$, $s = 0.1$, $\ell = u = 2$, and $\check{c} = \hat{c} = 1$. The corresponding control barriers are shown in Figure~\ref{fig:cb_alpha}. A higher drift increases the likelihood that the inventory process incurs the upper holding cost. To mitigate this, the DM should intervene earlier on the upper barrier, while delaying actions at the lower barrier. This allows the inventory more time to potentially return to the upper region, avoiding unnecessary intervention on the upper side.

Furthermore, an increase in $a$ causes the target region to shift upward, as previously discussed in Section~\ref{sec: general}, due to the condition $a + S(X_1 - X_2) = 0$. Taken together, these effects lead to an upward shift of the entire control barrier structure, as illustrated in Figure~\ref{fig:cb_alpha}.

\subsubsection{Comparative statics of holding cost}
Figure~\ref{fig:cb_cc} displays the control barrier configurations for varying lower holding costs, with $\check{c} \in \{1,2,3\}$ and parameters $\rho = 0.2$, $a = 0.2$, $b = 0.2$, $s = 0.1$, $\ell = u = 2$, and $\hat{c} = 1$. As the lower holding cost increases, it becomes more expensive to maintain a negative inventory level. Consequently, the DM should exert the lower control earlier to avoid incurring higher holding costs. Simultaneously, the DM should delay action on the upper control barrier, which is associated with a lower cost. Once the inventory reaches this upper region, not only is the holding cost relatively cheaper, but the probability of the inventory drifting into the costly negative region also decreases, thereby reducing the expected holding cost overall.

It is important to note that this control strategy applies when the inventory exhibits a positive trend, i.e., when the DM can afford to wait and learn. In contrast, if the inventory has a negative trend, it enters a control-dominant region where immediate intervention is required to prevent excessively high holding costs. Furthermore, in cases where the inventory level is excessively positive while the trend is negative, the control barriers remain relatively unchanged. This is because, in expectation, the inventory is likely to revert toward the lower region, which carries a higher cost. Hence, early intervention in such a situation, similar to the wait-and-learn region, is unnecessary.

\section{Conclusion}\label{sec:conclusion}

In this paper, we explore the application of singular control in inventory management under smooth ambiguity preference, where the model involves an unobservable parameter assumed to follow a Gaussian distribution. We establish a connection between the value function under smooth ambiguity and a forward-backward stochastic differential equation with quadratic growth. A verification theorem is provided in terms of the Hamilton-Jacobi-Bellman (HJB) equation to derive the optimal control policy. Additionally, we utilise the coordinate transform method to simplify the diffusion term associated with the unobservable parameter for practical use. We then apply the Markov chains approximation to numerically solve the HJB equation and conduct comparative statics to analyse the effects of smooth ambiguity on the optimal control policy, assuming the inventory level follows an arithmetic Brownian motion.

Our numerical study contributes a new perspective on singular control with both learning and ambiguity. We show that there are three critical control regions: the target, learning-dominant, and control-dominant regions, respectively. The first suggests a saddle point for which the decision-maker must take the maximum intervention to obtain the lowest running cost. This region also appears in the standard singular control. The second one represents the region where it is optimal to wait and learn the model evolution, while the last one imposes the region where learning should be inactive.
In the presence of ambiguity, these continuation regions become smaller, suggesting that the decision-maker is expected to act sooner and more often. However, this contraction occurs only on the learning-dominant side, as ambiguity arises through learning in the context of the smooth ambiguity framework. Consequently, in the control-dominant region, the control policy under smooth ambiguity remains indifferent, as learning in this region is suboptimal.

It is important to highlight that while the treatment of smooth ambiguity in continuous time is based on the framework of singular control, our theory can be extended to a broader range of problems in stochastic control, provided the HJB equations are applicable. For instance, impulse control represents a direct extension of singular control, where actions cause not only continuous reflection but also jumps of a certain magnitude. This issue becomes particularly relevant when there are fixed and proportional costs associated with inventory intervention. Another area that could consider ambiguity, aside from barrier control, is, for example, the McDonald-Siegel-type model, which presents a stylised optimal stopping problem in investment timing decisions. There is extensive literature on these topics, and we refer to \citet{pham2009continuous,Har13} for a comprehensive overview.


\newpage
\addcontentsline{toc}{section}{References}
\bibliography{sincon} 

\begin{thebibliography}{}

\bibitem[Aliprantis and Border, 2006]{Aliprantis2006-mo}
Aliprantis, C.~D. and Border, K.~C. (2006).
\newblock {\em Infinite dimensional analysis: A hitchhiker's guide}.
\newblock Springer, Berlin, Germany, 3 edition.

\bibitem[Archankul et~al., 2025]{archankul_thijssen_ferrari_hellmann_2025}
Archankul, A., Thijssen, J.~J., Ferrari, G., and Hellmann, T. (2025).
\newblock Singular control in an inventory model with ambiguity.
\newblock {\em Working Paper}.

\bibitem[Arrow et~al., 1951]{Arrow1951-mm}
Arrow, K.~J., Harris, T., and Marschak, J. (1951).
\newblock Optimal inventory policy.
\newblock {\em Econometrica}, 19(3):250--272.

\bibitem[Bar-Ilan and Sulem, 1995]{bar1995explicit}
Bar-Ilan, A. and Sulem, A. (1995).
\newblock Explicit solution of inventory problems with delivery lags.
\newblock {\em Mathematics of Operations Research}, 20(3):709--720.

\bibitem[Basei et~al., 2024]{Basei2024-pq}
Basei, M., Ferrari, G., and Rodosthenous, N. (2024).
\newblock Uncertainty over uncertainty in environmental policy adoption:
  Bayesian learning of unpredictable socioeconomic costs.
\newblock {\em J. Econ. Dyn. Control}, 161:104841.

\bibitem[Bather, 1966]{Bather1966-wi}
Bather, J.~A. (1966).
\newblock A continuous time inventory model.
\newblock {\em J. Appl. Probab.}, 3(2):538--549.

\bibitem[Bensoussan et~al., 2005]{bensoussan2005optimality}
Bensoussan, A., Liu, R., and Sethi, S.~P. (2005).
\newblock Optimality of an (s,s) policy with compound poisson and diffusion
  demands: A quasi-variational inequalities approach.
\newblock {\em SIAM journal on control and optimization}, 44(5):1650--1676.

\bibitem[Chen and Epstein, 2002]{ChEp02}
Chen, Z. and Epstein, L. (2002).
\newblock Ambiguity, risk, and asset returns in continuous time.
\newblock {\em Econometrica}, 70(4):1403--1443.

\bibitem[Cheng and Riedel, 2013]{Cheng2013-to}
Cheng, X. and Riedel, F. (2013).
\newblock Optimal stopping under ambiguity in continuous time.
\newblock {\em Math. Financ. Econ.}, 7(1):29--68.

\bibitem[Constantinides, 1976]{Constantinides1976-rf}
Constantinides, G.~M. (1976).
\newblock Stochastic cash management with fixed and proportional transaction
  costs.
\newblock {\em Manage. Sci.}, 22(12):1320--1331.

\bibitem[Crandall et~al., 1987]{crandall1987uniqueness}
Crandall, M.~G., Ishii, H., and Lions, P.-L. (1987).
\newblock Uniqueness of viscosity solutions of hamilton-jacobi equations
  revisited.
\newblock {\em Journal of the Mathematical Society of Japan}, 39(4):581--596.

\bibitem[Crandall and Lions, 1983]{crandall1983viscosity}
Crandall, M.~G. and Lions, P.-L. (1983).
\newblock Viscosity solutions of hamilton-jacobi equations.
\newblock {\em Transactions of the American mathematical society},
  277(1):1--42.

\bibitem[Dai and Yao, 2013a]{Dai2013-yz}
Dai, J.~G. and Yao, D. (2013a).
\newblock Brownian inventory models with convex holding cost, part 1:
  {Average-Optimal} controls.
\newblock {\em Stochastic Systems}, 3(2):442--499.

\bibitem[Dai and Yao, 2013b]{Dai2013-xb}
Dai, J.~G. and Yao, D. (2013b).
\newblock Brownian inventory models with convex holding cost, part 2:
  {Discount-Optimal} controls.
\newblock {\em Stochastic Systems}, 3(2):500--573.

\bibitem[De~Angelis, 2020]{De_Angelis2020-sr}
De~Angelis, T. (2020).
\newblock Optimal dividends with partial information and stopping of a
  degenerate reflecting diffusion.
\newblock {\em Finance Stoch.}, 24(1):71--123.

\bibitem[Eppen and Fama, 1969]{Eppen1969-js}
Eppen, G.~D. and Fama, E.~F. (1969).
\newblock Cash balance and simple dynamic portfolio problems with proportional
  costs.
\newblock {\em Int. Econ. Rev.}, 10(2):119--133.

\bibitem[Epstein and Zin, 1989]{Epstein1989-ow}
Epstein, L.~G. and Zin, S.~E. (1989).
\newblock Substitution, risk aversion, and the temporal behavior of consumption
  and asset returns: A theoretical framework.
\newblock {\em Econometrica}, 57(4):937--969.

\bibitem[Federico et~al., 2023]{Federico2023-pm}
Federico, S., Ferrari, G., and Rodosthenous, N. (2023).
\newblock Two-sided singular control of an inventory with unknown demand trend.
\newblock {\em SIAM J. Control Optim.}, 61(5):3076--3101.

\bibitem[Fleming and Soner, 2006]{fleming2006controlled}
Fleming, W.~H. and Soner, H.~M. (2006).
\newblock {\em Controlled Markov processes and viscosity solutions}, volume~25.
\newblock Springer Science \& Business Media.

\bibitem[Gilboa and Schmeidler, 1989]{GiSch89}
Gilboa, I. and Schmeidler, D. (1989).
\newblock Maxmin expected utility with non-unique prior.
\newblock {\em Journal of Mathematical Economics}, 18(2):141--153.

\bibitem[Gindrat and Lefoll, 2011]{gindrat2011smooth}
Gindrat, R. and Lefoll, J. (2011).
\newblock Smooth ambiguity aversion and the continuous-time limit.
\newblock {\em Available at SSRN 1690240}.

\bibitem[Hansen and Miao, 2018]{Hansen2018-yc}
Hansen, L.~P. and Miao, J. (2018).
\newblock Aversion to ambiguity and model misspecification in dynamic
  stochastic environments.
\newblock {\em Proc. Natl. Acad. Sci. U. S. A.}, 115(37):9163--9168.

\bibitem[Hansen and Miao, 2022]{Hansen2022-ob}
Hansen, L.~P. and Miao, J. (2022).
\newblock Asset pricing under smooth ambiguity in continuous time.
\newblock {\em Econom. Theory}, 74(2):335--371.

\bibitem[Hansen and Sargent, 2011]{Hansen2011-ix}
Hansen, L.~P. and Sargent, T.~J. (2011).
\newblock Robustness and ambiguity in continuous time.
\newblock {\em J. Econ. Theory}, 146(3):1195--1223.

\bibitem[Harrison, 2013]{Har13}
Harrison, J. (2013).
\newblock {\em Brownian Models of Performance and Control}.
\newblock Cambridge University Press, Cambridge.

\bibitem[Harrison and Taksar, 1983]{harrison1983instantaneous}
Harrison, J.~M. and Taksar, M.~I. (1983).
\newblock Instantaneous control of brownian motion.
\newblock {\em Mathematics of Operations research}, 8(3):439--453.

\bibitem[Harrison, 1978]{Michael_HARRISON1978-dq}
Harrison, M. (1978).
\newblock Optimal control of brownian storage system.
\newblock {\em Stochastic Processes and Their Applications}, 6:179--194.

\bibitem[Hellmann and Thijssen, 2018]{Hellmann2018-xs}
Hellmann, T. and Thijssen, J. J.~J. (2018).
\newblock Fear of the market or fear of the competitor? ambiguity in a real
  options game.
\newblock {\em Oper. Res.}, 66(6):1744--1759.

\bibitem[Johnson and Peskir, 2017]{johnson2017quickest}
Johnson, P. and Peskir, G. (2017).
\newblock Quickest detection problems for bessel processes.

\bibitem[Karatzas, 1983]{Karatzas1983-do}
Karatzas, I. (1983).
\newblock A class of singular stochastic control problems.
\newblock {\em Adv. Appl. Probab.}, 15(2):225--254.

\bibitem[Karatzas and Shreve, 1991]{karatzas1991brownian}
Karatzas, I. and Shreve, S. (1991).
\newblock {\em Brownian motion and stochastic calculus}, volume 113.
\newblock Springer Science \& Business Media.

\bibitem[Klibanoff et~al., 2005]{KlMaMu05}
Klibanoff, P., Marinacci, M., and Mukerji, S. (2005).
\newblock A smooth model of decision making under ambiguity.
\newblock {\em Econometrica}, 73(6):1849--1892.

\bibitem[Klibanoff et~al., 2009]{Klibanoff2009-zj}
Klibanoff, P., Marinacci, M., and Mukerji, S. (2009).
\newblock Recursive smooth ambiguity preferences.
\newblock {\em J. Econ. Theory}, 144(3):930--976.

\bibitem[Knight, 1921]{knight1921risk}
Knight, F.~H. (1921).
\newblock {\em Risk, uncertainty and profit}, volume~31.
\newblock Houghton Mifflin.

\bibitem[Kobylanski, 2000]{Kobylanski2000-zw}
Kobylanski, M. (2000).
\newblock Backward stochastic differential equations and partial differential
  equations with quadratic growth.
\newblock {\em aop}, 28(2):558--602.

\bibitem[Kocabıyıkoğlu et~al., 2024]{Kocabiyikoglu2024-ng}
Kocabıyıkoğlu, A., Önkal, D., Göğüş, C.~I., and Gönül, M.~S. (2024).
\newblock Newsvendor decisions under incomplete information: behavioural
  experiments on information uncertainty.
\newblock {\em IMA J. Manag. Math.}, 35(3):427--462.

\bibitem[Kushner and Dupuis, 2001]{kushner2001numerical}
Kushner, H. and Dupuis, P. (2001).
\newblock {\em Numerical Methods for Stochastic Control Problems in Continuous
  Time}.
\newblock Applications of mathematics. Springer.

\bibitem[Kushner and Martins, 1991]{kushner1991}
Kushner, H.~J. and Martins, L.~F. (1991).
\newblock Numerical methods for stochastic singular control problems.
\newblock {\em SIAM Journal on Control and Optimization}, 29(6):1443--1475.

\bibitem[Lee, 2022]{Lee2022-qb}
Lee, S.~Y. (2022).
\newblock Gibbs sampler and coordinate ascent variational inference: A
  set-theoretical review.
\newblock {\em Communications in Statistics - Theory and Methods},
  51(6):1549--1568.

\bibitem[Liptser and Shiryaev, 2013a]{liptser2013statisticsI}
Liptser, R.~S. and Shiryaev, A.~N. (2013a).
\newblock {\em Statistics of random processes: I. General theory}, volume~5.
\newblock Springer Science \& Business Media.

\bibitem[Liptser and Shiryaev, 2013b]{liptser2013statisticsII}
Liptser, R.~S. and Shiryaev, A.~N. (2013b).
\newblock {\em Statistics of random processes II: Applications}, volume~6.
\newblock Springer Science \& Business Media.

\bibitem[Ma and Aloysius, 2022]{Ma2022-ew}
Ma, S. and Aloysius, J. (2022).
\newblock Inventory control under different forms of uncertainty: Ambiguity and
  stochastic variability.
\newblock {\em Decis. Sci.}

\bibitem[Nishimura and Ozaki, 2007]{NiOz07}
Nishimura, K.~G. and Ozaki, H. (2007).
\newblock Irreversible investment and knightian uncertainty.
\newblock {\em Journal of Economic Theory}, 136(1):668--694.

\bibitem[Pham, 2009]{pham2009continuous}
Pham, H. (2009).
\newblock {\em Continuous-time stochastic control and optimization with
  financial applications}, volume~61.
\newblock Springer Science \& Business Media.

\bibitem[Skiadas, 2013]{Skiadas2013-eb}
Skiadas, C. (2013).
\newblock Smooth ambiguity aversion toward small risks and {Continuous-Time}
  recursive utility.
\newblock {\em Journal of Political Economy}, 121(4):775--792.

\bibitem[Thijssen, 2011]{Thijssen2011-ag}
Thijssen, J. J.~J. (2011).
\newblock Incomplete markets, ambiguity, and irreversible investment.
\newblock {\em J. Econ. Dyn. Control}, 35(6):909--921.

\bibitem[Trojanowska and Kort, 2010]{Trojanowska2010-zn}
Trojanowska, M. and Kort, P.~M. (2010).
\newblock The worst case for real options.
\newblock {\em J. Optim. Theory Appl.}, 146(3):709--734.

\bibitem[Vial, 1972]{Vial1972-ip}
Vial, J.-P. (1972).
\newblock A continuous time model for the cash balance problem.
\newblock {\em LIDAM Reprints CORE}.

\bibitem[Wang, 2004]{Wang2004-js}
Wang, B. (2004).
\newblock Singular control of stochastic linear systems with recursive utility.
\newblock {\em Syst. Control Lett.}, 51(2):105--122.

\bibitem[Zhang, 2017]{Zhang2017-jp}
Zhang, J. (2017).
\newblock {\em Backward Stochastic Differential Equations: From Linear to Fully
  Nonlinear Theory}, volume~86 of {\em Probability Theory and Stochastic
  Modelling}.
\newblock Springer New York.

\end{thebibliography}

\end{document}